\documentclass[a4paper,11pt,times,authoryears,reqno]{amsart}

\usepackage{geometry}
\usepackage{bm}
\usepackage{amssymb}
\usepackage{graphicx}
\usepackage{mathrsfs}
\usepackage{subfig}
\usepackage{float}
\usepackage{natbib}
\usepackage{enumerate}
\usepackage{booktabs}
\usepackage{adjustbox}
\usepackage[colorlinks,citecolor=blue]{hyperref}

\newtheorem{theorem}{Theorem}[section]
\newtheorem{lemma}[theorem]{Lemma}

\theoremstyle{definition}

\theoremstyle{remark}

\newtheorem{proposition}[theorem]{Proposition}

\numberwithin{equation}{section}
\numberwithin{equation}{section}

\numberwithin{equation}{section}

\def\bzero{{\mathbf 0}}

 \def\bC{{\mathbf C}}  
\def\bF{{\mathbf F}} \def\bS{{\mathbf S}}

 \def\bS{{\mathbf S}}  
  \def\bX{{\mathbf X}} \def\bY{{\mathbf Y}}

\def\bx{{\mathbf x}} \def\by{{\mathbf y}} 

\def\bbeta{{\boldsymbol{\beta}}}

 \def\bdelta{{\boldsymbol{\delta}}}

\def\bPhi{\boldsymbol{\Phi }}

  \def\bveps{{\boldsymbol{\varepsilon}}}

\def\hbbeta{\widehat{\boldsymbol \beta}}

 \def\hbeta{\widehat{\beta}}

\def\real{\mathop{{\rm I}\kern-.2em\hbox{\rm R}}\nolimits}

\def\1overn{\frac{1}{n}}

\def\bel{\begin{eqnarray}\label}  \def\eel{\end{eqnarray}}
\def\bes{\begin{eqnarray*}}  \def\ees{\end{eqnarray*}}

\begin{document}

\title{Liu-type Shrinkage Estimations in Linear Models}


\author{Bahad{\i}r Y\"{u}zba\c{s}{\i}}
\author{Yasin Asar}
\author{S. Ejaz Ahmed}

\address[Bahad{\i}r Y\"{u}zba\c{s}{\i}]{Department of Econometrics, Inonu University, Malatya, Turkey}
\email{b.yzb@hotmail.com}

\address[Yasin Asar]{Department of Mathematics-computer Sciences, Necmettin Erbakan University, Konya, Turkey }
\email{yasar@konya.edu.tr, yasinasar@hotmail.com }

\address[S. Ejaz Ahmed]{Department of Mathematics and Statistics,
Brock University, St. Catharines, Canada}
\email{sahmed5@brocku.ca}

\date{}

\begin{abstract}

In this study, we present the preliminary test, Stein-type and positive part Liu estimators in the linear models when the parameter vector $\bbeta$ is partitioned into two parts, namely, the main effects $\bbeta_1$ and the nuisance effects $\bbeta_2$ such that $\bbeta=\left(\bbeta_1, \bbeta_2 \right)$. We consider the case that a priori known or suspected set of the explanatory variables do not contribute to predict the response so that a sub-model may be enough for this purpose. Thus, the main interest is to estimate $\bbeta_1$ when $\bbeta_2$ is close to zero. Therefore, we conduct a Monte Carlo simulation study to evaluate the relative efficiency of the suggested estimators, where we demonstrate the superiority of the proposed estimators. \\
\medskip

\noindent \text{Keywords:} Liu Estimation, Pretest and Shrinkage Estimation, Penalty Estimation, Asymptotic and Simulation.

\end{abstract}

\maketitle

\section{Introduction}
Consider a linear regression model
\begin{equation}
y_{i}=\bx_{i}^{\top}\bm{\beta }+\varepsilon _{i},\ \ \
i=1,2,...,n,  \label{lin.mod}
\end{equation}%
where $y_{i}$'s are responses, $\bx_{i}=\left(
x_{i1},x_{i2},...,x_{ip}\right)^{\top}$ are observation points, $%
\bm{\beta }=\left( \beta _{1},\beta _{2},...,\beta _{p}\right)
^{\top}$ is a vector of unknown regression coefficients, $\varepsilon _{i}{}$'$s$
are unobservable random errors and the superscript $\left( ^{\top}\right) $
denotes the transpose of a vector or matrix. Further, $\bm{%
\varepsilon }=\left( \varepsilon _{1},\varepsilon _{2},...,\varepsilon
_{n}\right)^{\top}$ has a cumulative distribution function $\textnormal{F}\left(
\bm{\cdot }\right) $; $\textnormal{E}\left( \bm{\varepsilon }%
\right) =\bm{0}$ and $\textnormal{Var}\left( \bm{\varepsilon }\right)
=\sigma ^{2}\bm{I}_{n}$, where $\sigma ^{2}$ is finite and $%
\bm{I}_{n}$ is an identity matrix of dimension $n\times n.$ In this paper, we consider
that the design matrix has rank $p$ ($p\leq n$).

In a multiple linear regression model, it is usually assumed that the
explanatory variables are independent of each other. However, the multicollinearity problem arises when the explanatory variables are dependent. In this case, some biased estimations, such as shrinkage estimation, principal components
estimation (PCE), ridge estimation (\citet{Ho-Ke1970}), partial least squares (PLS) estimation Liu estimator (\citet{Liu1993}) and 
Liu-type estimator (\citet{Liu2003}) were proposed to improve the
least square estimation (LSE). To combat multicollinearity, \citet{yuzbasi-ahmed2016,yuzbasi-et-al2017} proposed the pretest and Stein-type ridge regression estimators for linear and partially linear models. 

In this study, we consider a linear regression model \eqref{lin.mod}
under the assumption of sparsity. Under this
assumption, the vector of coefficients $\bm{\beta }$ can be
partitioned as $\left( \bm{\beta }_{1},\bm{\beta }%
_{2}\right) $ where $\bm{\beta }_{1}$ is the coefficient vector for
main effects, and $\bm{\beta }_{2}$ is the vector for nuisance
effects or insignificant coefficients. We are essentially interested in the
estimation of $\bm{\beta }_{1}$ when it is reasonable that $%
\bm{\beta }_{2}$ is close to zero. The full model estimation may be subject to high variability and may not be easily interpretable. On the the other hand, a sub-model strategy may result with an under-fitted model with large bias. For this reason, we consider pretest and shrinkage strategy to control the magnitude of the bias. %
Also, \citet{ahmed2014} gave a detailed definition of shrinkage estimation techniques in regression models.  


The paper is organized as follows. The full and sub-model estimators based on Liu regression are given in Section 2. Moreover, the pretest, shrinkage estimators and penalized estimations are also given in this section. The asymptotic properties of the pretest and shrinkage estimators estimators are obtained in Section 3. The design and the results of a Monte Carlo simulation study including a comparison with other penalty estimators are given in Section 4. A real data example is given for illustrative purposes in Section 5.  The concluding remarks are presented in Section 6.

\section{Estimation Strategies}
The ridge estimator firstly proposed by \citet{Ho-Ke1970} can be obtained from the
following model
\begin{equation*}
\label{ridge1}
\bY =\bX\bbeta+\bm{\varepsilon}\ \ \text{%
subject to }\bm{\beta }'\bm{\beta}%
\bm{\leq }\phi,
\end{equation*}%
where $\by=\left( y_{1},\ldots,y_{n}\right)^{\top}$, $\bX=\left( \bx_{1},\ldots,\bx_{n}\right)^{\top}$ and $\phi $ is inversely proportional to $\lambda ^{R}$, which is equal to
\begin{equation*}
\underset{\bm{\beta }}{\arg \min }\left\{
\sum\nolimits_{i=1}^{n}\left( y_{i}-\bx_{i}^{\bm{\top }}%
\bm{\beta }\right) ^{2}+\lambda ^{R}\sum\nolimits_{j=1}^{p}\beta
_{j}^{2}\right\} .  \label{ridge_RSS}
\end{equation*}
It yields%
\begin{equation*}
\bm{\widehat{\beta}}^{\textnormal{RFM}}=\left(\bX^{\top}\bX%
+\lambda ^{R}\bm{I}_{p}\right) ^{-1}\bX^{\top}%
\bY,  \label{ridge_ue}
\end{equation*}%
where $\bm{\widehat{\beta}}^{\textnormal{RFM}}$ is called a ridge full model estimator and $%
\bY = \left( y_{1},y_{2},...,y_{n}\right) ^{\top}$. If $\lambda
^{R}=0,$ then $\bm{\widehat{\beta}}^{\textnormal{RFM}}$ is the LSE estimator, and $%
\lambda ^{R}=\infty ,$ then $\bm{\widehat{\beta}}^{\textnormal{RFM}}=\bm{0%
}.$

\citet{Liu1993} proposed a new biased estimator (LFM) by augmenting $d \bm{\widehat{\beta}}^{\textnormal{LSE}}=\beta+\epsilon'$ to \eqref{lin.mod} such that
\begin{equation*}\label{LFM}
\bm{\widehat{\beta}}^{\textnormal{LFM}}=\left(\bX^{\top}\bX%
+\bm{I}_{p}\right) ^{-1}\left(\bX^{\top}\bX%
+d^{\textnormal{L}}\bm{I}_{p}\right)\bm{\widehat{\beta}}^{\textnormal{LSE}}
\end{equation*} 
where $0<d^{\textnormal{L}}<1$. The advantage of $\bm{\widehat{\beta}}^{\textnormal{LFM}}$ over $\bm{\widehat{\beta}}^{\textnormal{RFM}}$ is that $\bm{\widehat{\beta}}^{\textnormal{LFM}}$ is a linear function of $d^{\textnormal{L}}$ (\cite{Liu1993}). When $d^{\textnormal{L}}=1$, we get $\bm{\widehat{\beta}}^{\textnormal{LFM}}=\bm{\widehat{\beta}}^{\textnormal{LSE}}$. Many researcher has been considered Liu estimator so far. Among them, \citet{Liu2003}, \citet{akdeniz2003}, \citet{Hu-Wi2006}, \citet{Sa-Ki1993} and \citet{Ki2012} are notable.

We let $\bX=\left( \bX_{\bm{1}},%
\bX_{\bm{2}}\right) $, where $\bX_{\bm{%
1}}$ is an $n\times p_{1}$ sub-matrix containing the regressors of interest
and $\bX_{\bm{2}}$ is an $n\times p_{2}$ sub-matrix that
may or may not be relevant in the analysis of the main regressors. Similarly,
$\bm{\beta }=\left( \bm{\beta }_{1}^{\top},\bm{%
\beta }_{2}^{\top}\right) ^{\top}$ be the vector of parameters, where $%
\bm{\beta }_{1}$ and $\bm{\beta }_{2}$ have dimensions $%
p_{1} $ and $p_{2}$, respectively, with $p_{1}+p_{2}=p$, $p_{i}$ $\geq 0$
for $i=1,2$.

A sub-model or restricted model is defined as:
\begin{equation*}
\bY=\bX\bbeta+\bm{\varepsilon }\text{ \ \
subject to }\bm{\beta }^{\top}\bm{\beta }\text{ }%
\leq \phi \text{ and }\bm{\beta_2}=\bold{0},
\label{rid.re.mod}
\end{equation*}%
then we have the following restricted linear regression model
\begin{equation}
\bY=\bX_{1}\bm{\beta }_{1}+\bm{%
\varepsilon }\text{ subject to }\bm{\beta }_{1}^{\top}\bm{%
\beta }_{1}\leq \phi.  \label{res.mod1}
\end{equation}%
We denote $\bm{\widehat{\beta}}_{1}^{\textnormal{RFM}}$ as the full model or
unrestricted ridge estimator of $\bm{\beta }_{1}$ is given by
\begin{equation*}
\bm{\widehat{\beta}}_{1}^{\textnormal{RFM}}=\left( \bX_{1}^{\top}%
\bm{M}_{2}^{R}\bX_{1}+\lambda ^{R}\bm{I}%
_{p_{1}}\right) ^{-1}\bX_{1}^{\top}\bm{M}_{2}^{R}%
\bY,
\end{equation*}%
where $\bm{M}_{2}^{R}=\bm{I}_{n}-\bX_{2}\left(
\bX_{2}^{\top}\bX_{2}+\lambda ^{R}\bm{I}%
_{p_{2}}\right) ^{-1}\bX_{2}^{\top}.$ For model \eqref{res.mod1}, the sub-model or restricted estimator $\bm{\widehat{\beta}}%
_{1}^{\textnormal{RSM}}$ of $\bm{\beta }_{1}$ has the form
\begin{equation*}
\bm{\widehat{\beta}}_{1}^{\textnormal{RSM}}=\left( \bX_{1}^{\top}%
\bX_{1}+\lambda _{1}^{R}\bm{I}_{p_{1}}\right) ^{-1}%
\bX_{1}^{\top}\bY,
\end{equation*}
where $\lambda _{1}^{R}$ is ridge parameter for sub-model estimator
$\bm{\widehat{\beta}}_{1}^{\textnormal{RSM}}$.

Similarly, we introduce the full model estimator or unrestricted Liu estimator $\bm{\widehat{\beta}}_{1}^{\textnormal{LFM}}$ as follows:

\begin{equation*}\label{LFM_1}
\bm{\widehat{\beta}}_{1}^{\textnormal{LFM}}=\left( \bX_{1}^{\top}%
\bm{M}_{2}^{\textnormal{L}}\bX_{1}+\bm{I}%
_{p_{1}}\right) ^{-1}\left( \bX_{1}^{\top}%
\bm{M}_{2}^{\textnormal{L}}\bX_{1}+d^{\textnormal{L}}\bm{I}%
_{p_{1}}\right)\bm{\widehat{\beta}}_{1}^{\textnormal{LSE}},
\end{equation*}%
where $\bm{M}_{2}^{\textnormal{L}}=\bm{I}_{n}-\bX_{2}\left(\bX_{2}'\bX_{2}+\bm{I}_{p_{2}}\right) ^{-1}\left(\bX_{2}^{\top}\bX_{2}+d^{\textnormal{L}}\bm{I}_{p_{2}}\right)\bX_{2}^{\top}.$ The sub-model Liu estimator $\bm{\widehat{\beta}}_{1}^{\textnormal{LSM}}$ is defined as
\begin{equation*}\label{LSM_1}
\bm{\widehat{\beta}}_{1}^{\textnormal{LSM}}=\left( \bX_{1}^{\top}\bX_{1}+\bm{I}%
_{p_{1}}\right) ^{-1}\left( \bX_{1}^{\top}\bX_{1}+d_1^{\textnormal{L}}\bm{I}%
_{p_{1}}\right)\bm{\widehat{\beta}}_{1}^{\textnormal{LSE}}
\end{equation*}%
where $0<d_1^{L}<1$ and $ \bm{\widehat{\beta}}_{1}^{\textnormal{LSE}} = \left(\bX_{1}^{\top}\bX_{1}\right) ^{-1}\bX_{1}^{\top}\bY.$

Generally speaking, $\bm{\widehat{\beta}}_{1}^{\textnormal{LSM}}$ performs better
than $\bm{\widehat{\beta}}_{1}^{\textnormal{LFM}}$ when $\bm{\beta }_{2}$
is close to zero. However, for $\bm{\beta }_{2}$ away from the zero, $%
\bm{\widehat{\beta}}_{1}^{\textnormal{LSM}}$ can be inefficient. But, the
estimate $\bm{\widehat{\beta}}_{1}^{\textnormal{LFM}}$ is consistent for
departure of $\bm{\beta }_{2}$ from zero.

The idea of penalized estimation was introduced by \citet{FF1993}.
They suggested the notion of bridge regression as given in \ref{bridge.mod}.
For a given penalty function $\pi \left( \bm{\cdot }\right) $ and
tuning parameter that controls the amount of shrinkage $\lambda $, bridge
estimators are estimated by minimizing the following penalized least square criterion
\begin{equation}
\sum_{i=1}^{n}\left( y_{i}-\bx_{i}^{\top }\bm{\beta }%
\right) ^{2}+\lambda \pi \left(\beta\right),  \label{bridge.mod}
\end{equation}%
where $\pi \left(\beta\right)$ is $\sum_{j=1}^{p}\left\vert \beta_{j}\right\vert
^{\gamma },\text{ }\gamma >0$. This penalty function
bounds the $L_{\gamma }$ norm of the parameters. 

\subsection{Pretest and Shrinkage Liu Estimation}
The pretest is a combination of $\bm{\widehat{\beta}}_{1}^{\textnormal{LFM}}$
and $\bm{\widehat{\beta}}_{1}^{\textnormal{LSM}}$ through an indicator function $%
I\left( \mathscr{L}_{n}\leq c_{n,\alpha }\right) ,$ where $\mathscr{L}_{n}$
is appropriate test statistic to test $H_{0}:\bm{\beta }_{2} = %
\bm{0}$ versus $H_{A}:\bm{\beta }_{2} \neq \bm{ 0.}$
Moreover, $c_{n,\alpha }$ is an $\alpha -$level critical value using the
distribution of $\mathscr{L}_{n}.$ We define the test statistic as follows:
\begin{equation*}
\mathscr{L}_{n}=\frac{n}{\widehat{\sigma}^{2}}\left( \bm{\widehat{%
\beta}}^{\textnormal{LSE}}_{2}\right) ^{\top}\bX_{2}^{\top}\bm{M}_{1}%
\bX_{2}\left( \bm{\widehat{\beta}}^{\textnormal{LSE}}_{2}\right) ,
\end{equation*}
where
$\widehat{\sigma}^{2}=\frac{1}{n-p}(\bY-\bX\bm{%
\widehat{\beta}}^{\textnormal{LFM}})^{\top}(\bY-\bX\bm{%
\widehat{\beta}}^{\textnormal{LFM}})$ is consistent estimator of $\sigma^{2}$,
$\bm{M}_{1}=\bm{I}_{n}-\bX_{1}\left(
\bX_{1}^{\top}\bX_{1}\right) ^{-1}\bX%
_{1}^{\top}$ and $\bm{\widehat{\beta}}^{\textnormal{LSE}}_{2}=\left( \bX%
_{2}^{\top}\bm{M}_{1}\bX_{2}\right) ^{-1}\bX%
_{2}^{\top}\bm{M}_{1}\bY$. Under $H_{0}$, the test
statistic $\mathscr{L}_{n}$ follows chi-square distribution with $p_{2}$
degrees of freedom for large $n$ values. The pretest test Liu
regression estimator $\bm{\widehat{\beta}}_{1}^{\textnormal{LPT}}$ of $\bm{\beta }_{1}$ is defined by%
\begin{equation*}
\bm{\widehat{\beta}}_{1}^{\textnormal{LPT}}=\bm{\widehat{\beta}}%
_{1}^{\textnormal{LFM}}-\left( \bm{\widehat{\beta}}_{1}^{\textnormal{LFM}}-\bm{%
\widehat{\beta}}_{1}^{\textnormal{LSM}}\right) I\left( \mathscr{L}_{n}\leq c_{n,\alpha
}\right),
\end{equation*}%
where $c_{n,\alpha }$ is an $\alpha -$ level critical value.

The shrinkage or Stein-type Liu regression estimator $\bm{\widehat{\beta}}_{1}^{\textnormal{LS}}$ of $\bm{\beta
}_{1}$ is defined by
\begin{equation*}
\bm{\widehat{\beta}}_{1}^{\textnormal{LS}}=\bm{\widehat{\beta}}%
_{1}^{\textnormal{LSM}}+\left( \bm{\widehat{\beta}}_{1}^{\textnormal{LFM}}-\bm{%
\widehat{\beta}}_{1}^{\textnormal{LSM}}\right) \left( 1-(p_{2}-2)\mathscr{L}%
_{n}^{-1}\right) \text{, }p_{2}\geq 3.
\end{equation*}

The estimator $\bm{\widehat{\beta}}_{1}^{\textnormal{LS}}$ is general form of the
Stein-rule family of estimators where shrinkage of the base estimator is
towards the restricted estimator $\bm{\widehat{\beta}}_{1}^{\textnormal{LSM}}$. The
Shrinkage estimator is pulled towards the restricted estimator when the
variance of the unrestricted estimator is large. Also, we can say that $\bm{\widehat{%
\beta}}_{1}^{RS}$ is the smooth version of $\bm{\widehat{\beta}}%
_{1}^{\textnormal{LPT}}$.

The positive part of the shrinkage Liu regression estimator $\bm{\widehat{\beta }}_{1}^{\textnormal{LPS}}$ of $%
\bm{\beta }_{1}$ can be defined by%
\begin{equation*}
\bm{\widehat{\beta }}_{1}^{\textnormal{LPS}}=\bm{\widehat{\beta }}%
_{1}^{\textnormal{LSM}}+\left( \bm{\widehat{\beta }}_{1}^{\textnormal{LFM}}-\bm{%
\widehat{\beta }}_{1}^{\textnormal{LSM}}\right) \left( 1-(p_{2}-2)\mathscr{L}%
_{n}^{-1}\right) ^{+},
\end{equation*}%
where $z^{+}=max(0,z)$.

\subsubsection{Lasso strategy} For $\gamma =1,$ we obtain the $L_{1}$ penalized least squares
estimator, which is commonly known as Lasso (least absolute shrinkage and selection
operator) 
\begin{equation*}
\bm{\widehat{\beta }}^{\textnormal{Lasso}}\underset{\bm{\beta }}{=\arg
\min }\left\{ \sum\nolimits_{i=1}^{n}\left( y_{i}-\bx_{i}^{%
\bm{\top }}\bm{\beta }\right) ^{2}+\lambda
\sum_{j=1}^{p}\left\vert \beta _{j}\right\vert \right\}.
\end{equation*}%
where the parameter $\lambda \geq 0$ controls the amount of shrinkage, see \citet{lasso} for details.
Lasso is a popular estimator in order to provide simultaneous estimation and variable selection.
\subsubsection{Adaptive Lasso strategy} The adaptive Lasso estimator is defined as%
\begin{equation*}
\bm{\widehat{\beta }}^{\textnormal{aLasso}}\underset{\bm{\beta }}{=\arg
\min }\left\{ \sum\nolimits_{i=1}^{n}\left( y_{i}-\bx_{i}^{%
\bm{\top }}\bm{\beta }\right) ^{2}+\lambda \sum_{j=1}^{p}%
\widehat{\xi }_{j}\left\vert \hbeta _{j}\right\vert \right\},  \label{aLasso}
\end{equation*}%
where the weight function is
\begin{equation*}
\widehat{\xi }_{j}=\frac{1}{\left\vert {\beta_j }^{\ast
}\right\vert ^{\gamma }};\text{ }\gamma >0.
\end{equation*}%
The ${\hbeta_j }^{\ast }$ is the jth component of a root--n consistent estimator of ${\beta }$. For computational
details we refer to \citet{zou2006}.

\subsubsection{SCAD strategy} The smoothly clipped absolute deviation (SCAD) is proposed by
\citet{FanLi2001}. The SCAD penalty is given by
\begin{equation*}
J_{\alpha ,\lambda }\left( x\right) =\lambda \left\{ I\left( \left\vert
x\right\vert \leq \lambda \right) +\frac{\left( \alpha \lambda -\left\vert
x\right\vert \right) _{+}}{\left( \alpha -1\right) \lambda }I\left(
\left\vert x\right\vert >\lambda \right) \right\} ,\text{ \ }x\geq 0,
\label{SCAD_pen}
\end{equation*}%
for some $\alpha >2$ and $\lambda>0.$ Hence, the SCAD estimation is given
by
\begin{equation*}
\bm{\widehat{\beta }}^{\textnormal{SCAD}}\underset{\bm{\beta }}{=\arg
\min }\left\{ \sum\nolimits_{i=1}^{n}\left( y_{i}-\bx_{i}^{%
\bm{\top }}\bm{\beta }\right) ^{2}+\lambda
\sum_{j=1}^{p}J_{\alpha ,\lambda }\left\Vert \beta _{j}\right\Vert
_{1}\right\} ,
\end{equation*}%
where $\left\Vert \cdot \right\Vert _{1}$ denotes $L_{1}$ norm.

 For estimation strategies based on $\gamma= 2$, we establish some useful asymptotic results in the following section.

\section{Asymptotic Analysis}
We consider a sequence of local alternatives $\left\{ K_{n}\right\} $ given
by
\begin{equation*}
K_{n}:\bbeta_{2}=\bbeta_{2\left( n\right) }=\frac{%
\boldsymbol{\kappa}}{\sqrt{n}},
\end{equation*}%
where $\boldsymbol{\kappa}=\left(\kappa_{1},\kappa_{2},...,\kappa_{p_{2}}\right) ^{\top}$ is a fixed vector.
The asymptotic bias of an estimator $\hbbeta%
_{1}^{\ast }$ is defined as
\begin{eqnarray*}
\mathcal{B}\left( \hbbeta_{1}^{\ast }\right) =\mathcal{E} \underset{%
n\rightarrow \infty }{\lim }\left\{\sqrt{n}\left( \hbbeta_{1}^{\ast }-%
\bbeta_{1}\right) \right\},
\end{eqnarray*}
the asymptotic covariance of an estimator $\hbbeta _{1}^{\ast }$ is given by
\begin{eqnarray*}
\boldsymbol{\Gamma} \left( \hbbeta_{1}^{\ast} \right) =\mathcal{E} \underset{n\rightarrow \infty }{\lim }\left\{n\left( \hbbeta_{1}^{\ast} -\bbeta_{1}\right) \left( \hbbeta_{1}^{\ast} -\bbeta_{1}\right) ^{\top}\right\},
\end{eqnarray*}
and by using asymptotic the covariance matrix $\boldsymbol{\Gamma}\left( \hbbeta_{1}^{\ast} \right)$, the asymptotic risk of an estimator $\hbbeta_{1}^{\ast }$ is given by 
\begin{eqnarray*}
\mathcal{R}\left( \hbbeta_{1}^{\ast }\right)=\textnormal{tr}\left( \boldsymbol{W}\boldsymbol{\Gamma}\left( \hbbeta_{1}^{\ast} \right)\right),
\end{eqnarray*}
where $\boldsymbol{W}$ is a positive definite matrix of weights with dimensions of $p\times p$, and $\hbbeta_{1}^{\ast }$ is one of the suggested estimators:

We consider the following regularity conditions in order to evaluate the asymptotic properties of the estimators.

 \begin{enumerate}[(i)]
   \item $\frac{1}{n}\underset{1\leq i\leq n}{\max }%
\bx_{i}^{\top}(\bX^{\top}\bX)^{-1}\bx_{i}\rightarrow 0$ as $n\rightarrow
\infty ,$ where $\bx_{i}^{\top}$ is the $i$th row of $%
\bX$
   \item $\underset{n\rightarrow \infty }{\lim } n^{-1}(\bX^{\top}
\bX) = \underset{n\rightarrow \infty }{\lim } \bC_n = \bC, \text{for finite } \bC.$

\item $\underset{n\rightarrow \infty }{\lim } \bF_n(d)= \boldsymbol{\bF_d}, \text{for finite } \bF_d$ where $\bF_n(d) = \left(\bC_n +\bm{I}_p \right)^{-1}\left(\bC_n +d\bm{I}_p \right)$ \\ and $\bF_d = \left(\bC+\bm{I}_p \right)^{-1}\left(\bC +d\bm{I}_p \right)$.
 \end{enumerate}
 
\begin{theorem} If $0<d<1$ and $\bC$ is non-singular, then%
\begin{equation*}
\sqrt{n}\left( \hbbeta^{\rm LFM}-\bbeta%
\right)\sim\mathcal{N}_{p}\left( -(1-d)\left(\bC+\bm{I}_{p}\right)%
^{-1}\bbeta\text{, }\sigma ^{2}\bS\right)
\end{equation*}
where $ \bS=\bF_d\bC^{-1}\bF_d^{\top}$.
\label{teo1}
\end{theorem}

\begin{proof}
Since $\hbbeta^{\rm LFM}$ is a linear function of $\hbbeta^{\rm LSE}$, it is asymptotically normally distributed.
\begin{eqnarray*}
\mathcal{E}\left( \sqrt{n}\left(\hbbeta^{\rm LFM}- \bbeta \right) \right) &=& \mathcal{E} \underset{%
n\rightarrow \infty }{\lim }\left\{\sqrt{n}\left(\left(\bC
+\bm{I}_{p}\right) ^{-1}\left(\bC
+d\bm{I}_{p}\right)\hbbeta^{\rm LSE}-\bbeta\right) \right\} \nonumber \\
&=& \mathcal{E} \underset{%
n\rightarrow \infty }{\lim }\left\{\sqrt{n}\left[\left( \bm{I}_{p}-\left(\bC
+\bm{I}_{p}\right) ^{-1}+d\left(\bC
+d\bm{I}_{p}\right)\right) \hbbeta^{\rm LSE}-\bbeta \right]\right\} \nonumber \\
&=& -(1-d)\left(\bC+d\bm{I}_{p}\right)\bbeta
\end{eqnarray*}

\begin{eqnarray*}
\boldsymbol{\Gamma} \left( \hbbeta^{\rm LFM}\right) &=&\mathcal{E} \underset{n\rightarrow \infty }{\lim }\left\{n\left( \hbbeta^{\rm LFM} -\bbeta\right) \left( \hbbeta^{\rm LFM} -\bbeta \right) ^{\top}\right\} \nonumber \\
&=& \mathcal{E} \underset{n \rightarrow \infty }{\lim } n\left\{ \left(\bF_d\hbbeta^{\rm LSE}-\bbeta\right) \left(\bF_d\hbbeta^{\rm LSE}-\bbeta\right)^{\top} \right\}\nonumber \\
&=& \mathcal{E} \underset{n \rightarrow \infty }{\lim } n\left\{ \left(\left(\bF_d-\bm{I}_{p}\right)\bbeta+\bF_d \bC ^{-1}\bX^{\top}\bveps \right) \left(\left(\bF_d-\bm{I}_{p}\right)\bbeta+\bF_d \bC ^{-1}\bX^{\top}\bveps \right) ^{\top}\right\}\nonumber \\
&=& \sigma^2 \bF_d\bC^{-1}\bF_d^{\top}
\end{eqnarray*}
Hence, the Theorem is proven.
\end{proof}

\begin{proposition} 
Let $\vartheta _{1} = \sqrt{n}\left(\hbbeta_{1}^{\rm LFM}-%
\bbeta_{1}\right)$, $\vartheta _{2} =\sqrt{n}\left( \hbbeta_{1}^{\rm LSM}-%
\bbeta_{1}\right)$ and $\vartheta _{3} =\sqrt{n}\left( \hbbeta_{1}^{\rm LFM}-%
\hbbeta_{1}^{\rm LSM}\right)$. Under the foregoing assumptions, Theorem~\ref{teo1} and the local alternatives $\left\{ K_{n}\right\}$ as $n\rightarrow \infty$ we have


\label{pro1}

$$\left(
\begin{array}{c}
\vartheta _{1} \\
\vartheta _{3}%
\end{array}%
\right) \sim\mathcal{N}\left[ \left(
\begin{array}{c}
-\boldsymbol{\mu }_{11.2} \\
\boldsymbol{\delta }%
\end{array}%
\right) ,\left(
\begin{array}{cc}
\sigma ^{2}\bS_{11.2}^{-1} & \boldsymbol{\Phi } \\
\boldsymbol{\Phi } & \boldsymbol{\Phi }%
\end{array}%
\right) \right],$$


$$\left(
\begin{array}{c}
\vartheta _{3} \\
\vartheta _{2}%
\end{array}%
\right) \sim\mathcal{N}\left[ \left(
\begin{array}{c}
\boldsymbol{\delta } \\
-\boldsymbol{\gamma }%
\end{array}%
\right) ,\left(
\begin{array}{cc}
\boldsymbol{\Phi } & \boldsymbol{0} \\
\boldsymbol{0} & \sigma ^{2}\bS_{11}^{-1}%
\end{array}%
\right) \right],$$ \\
where $\bC$ = $\left(
\begin{array}{cc}
\bC_{11} & \bC_{12} \\
\bC_{21} & \bC_{22}%
\end{array}%
\right)$, 
$\bS$ = $\left(
\begin{array}{cc}
\bS_{11} & \bS_{12} \\
\bS_{21} & \bS_{22}%
\end{array}%
\right)$,
\\$\boldsymbol{\gamma }=-(\boldsymbol{\mu }_{11.2}+\boldsymbol{\delta })$
and $\boldsymbol{\delta }$ $=\left(\bC_{11}+\bm{I}_{p_1}\right)^{-1}\left( \bC_{11}+d\bm{I}_{p_1}\right)\bC_{12}%
\boldsymbol{\kappa}$, $\boldsymbol{\Phi }=\sigma^{2}\bF_d^{11}\bC_{12}\bS%
_{22.1}^{-1}\bC_{21} \bF_d^{11}$, where $\bS_{22.1}=\bS_{22}-\bS_{21}\bS_{11}^{-1}\bS_{12}$, 
$\bF_d^{11}=\left(\bC_{11}+\bm{I}_{p_1}\right)^{-1} \left( \bC_{11}+d\bm{I}_{p_1}\right)$ and $\boldsymbol{\mu }=-(1-d)\left(\bC+\bm{I}\right)^{-1}\boldsymbol{%
\beta}$ = $\left(
\begin{array}{c}
\boldsymbol{\mu }_{1} \\
\boldsymbol{\mu }_{2}%
\end{array}%
\right)$ and $\boldsymbol{\mu }_{11.2}=\boldsymbol{\mu }_{1}-\bC_{12}%
\bC_{22}^{-1} \left( \left( \bbeta_{2}-\boldsymbol{\kappa}%
\right) -\boldsymbol{\mu }_{2}\right)$ such that $\boldsymbol{\mu }_{11.2}$ is the mean of conditional distribution of $\bbeta_1$, given $\bbeta_2=\bzero_{p_2}$ and $\sigma^2\bS_{11.2}^{-1}$ ($\bS_{11.2}=\bS_{11}-\bS_{12}\bS_{22}^{-1}\bS_{21}$) is the covariance matrix.
\end{proposition}

\begin{proof}

We make use of $\bm{%
\widetilde{y}}=\by-\bX_{2}\hbbeta_{2}^{\rm LFM}$ to obtain $\bPhi$ as follows
\begin{eqnarray}
\lefteqn{\hbbeta_{1}^{\rm LFM} }\nonumber \\
&=&\left( \bX_{1}^{\top }\bX_{1}+
\boldsymbol{I}_{p_{1}}\right) ^{-1}\left( \bX_{1}^{\top }\bX_{1}+d^L%
\boldsymbol{I}_{p_{1}}\right)\bX_{1}^{\top }\boldsymbol{%
\widetilde{y}}  \nonumber \\
&=&\left( \bX_{1}^{\top }\bX_{1}+
\boldsymbol{I}_{p_{1}}\right) ^{-1}\left( \bX_{1}^{\top }\bX_{1}+d^L%
\boldsymbol{I}_{p_{1}}\right)\bX_{1}^{\top }\by%
-\left( \bX_{1}^{\top }\bX_{1}+
\boldsymbol{I}_{p_{1}}\right) ^{-1}\left( \bX_{1}^{\top }\bX_{1}+d^L%
\boldsymbol{I}_{p_{1}}\right)\bX_{1}^{\top }\bX_{2}%
\hbbeta_{2}^{\rm LFM}  \nonumber \\
&=&\hbbeta_{1}^{\rm LSM}-\left( \bX_{1}^{\top }\bX_{1}+
\boldsymbol{I}_{p_{1}}\right) ^{-1}\left( \bX_{1}^{\top }\bX_{1}+d^L%
\boldsymbol{I}_{p_{1}}\right)%
\bX_{1}^{\top }\bX_{2}\hbbeta%
_{2}^{\rm LFM} \text{.}  \label{full_sub-1}
\end{eqnarray}%

Now, under the local alternatives $\left\{ K_{n}\right\}$, using \ref{full_sub-1}, we compute $\boldsymbol{\Phi }$ as follows:

\begin{eqnarray*}
\boldsymbol{\Phi } &=& Cov\left( \hbbeta_{1}^{\rm LFM}-%
\hbbeta_{1}^{\rm LSM}\right)  \\
&=&\mathcal{E}\left[ \left( \hbbeta_{1}^{\rm LFM} -\hbbeta_{1}^{\rm LSM}\right) \left(
\hbbeta_{1}^{\rm LFM}-\hbbeta%
_{1}^{\rm LSM}\right) ^{\top }\right]  \\
&=&\mathcal{E}\left[ \left( \left( \bC_{11}+\bm{I}_{p_1}\right)^{-1}\left( \bC_{11}+d\bm{I}_{p_1}\right)\bC_{12}\boldsymbol{%
\widehat{\beta}}_{2}^{\rm LFM}\right) \left( \left( \bC_{11}+\bm{I}_{p_1}\right)^{-1}\left( \bC_{11}+d\bm{I}_{p_1}\right)\bC_{12}\boldsymbol{%
\widehat{\beta}}_{2}^{\rm LFM}\right) ^{\top }\right]  \\
&=&\left( \bC_{11}+\bm{I}_{p_1}\right)^{-1}\left( \bC_{11}+d\bm{I}_{p_1}\right)\bC_{12}\mathcal{E}\left[ \boldsymbol{\widehat{\beta}%
}_{2}^{\rm LFM}\left( \hbbeta_{2}^{\rm LFM}\right) ^{\top }\right]%
\bC_{21} \left( \bC_{11}+d\bm{I}_{p_1}\right)\left( \bC_{11}+\bm{I}_{p_1}\right)^{-1} \\
&=&\sigma^{2}\bF_d^{11}{C}_{12}\bS_{22.1}^{-1}\bC_{21} \bF_d^{11}.
\end{eqnarray*}
where $\bF_d^{11}=\left( \bC_{11}+\bm{I}_{p_1}\right)^{-1}\left( \bC_{11}+d\bm{I}_{p_1}\right)$.

Due to \citet{J-W2014}, it is easy to obtain the distribution of $\vartheta_1$ as follows (see page 160, Result 4.6):

\begin{equation*}
\sqrt{n}\left(\hbbeta_{1}^{\rm LFM}-\bbeta_1%
\right)\sim\mathcal{N}_{p_1}\left( -\boldsymbol{\mu}_{11.2} \text{, }\sigma ^{2}\bS^{-1}_{11.2}\right),
\end{equation*}
Since $\vartheta_2$ and $\vartheta_3$ are linear functions of $\hbbeta^{\textnormal{LSE}}$, they are also asymptotically normally distributed.

Hence, the asymptotic distributions of the vectors $\vartheta_2$ and $\vartheta_3$ are easily obtained as follows:
\begin{eqnarray*}
\sqrt{n}\left(\hbbeta_{1}^{\rm LSM}-\bbeta_1%
\right)&\sim&\mathcal{N}_{p_1}\left(-\boldsymbol{\gamma}\text{, }\sigma ^{2}\bS^{-1}_{11}\right), \cr \\ 
\sqrt{n}\left(\hbbeta_{1}^{\rm LFM}-\hbbeta_{1}^{\rm LSM}\right)&\sim&\mathcal{N}_{p_1}\left(\bdelta\text{, }\boldsymbol{\Phi}\right).
\end{eqnarray*}
\end{proof}
The following lemma will be used in some of the proofs.
\begin{lemma} \label{lem_JB}
Let $\bX$ be $q-$dimensional normal vector distributed as $%
N\left( \boldsymbol{\mu }_{x},\boldsymbol{\Sigma }%
_{q}\right) ,$ then, for a measurable function of of $\varphi ,$ we have
\begin{align*}
\mathcal{E}\left[ \bX\varphi \left( \bX^{\top}\bX%
\right) \right] =&\boldsymbol{\mu }_{x}\mathcal{E}\left[ \varphi \chi _{q+2}^{2}\left(
\Delta \right) \right] \\
\mathcal{E}\left[ \boldsymbol{XX}^{\top}\varphi \left( \bX^{\top}%
\bX\right) \right] =&\boldsymbol{\Sigma }_{q}\mathcal{E}\left[ \varphi
\chi _{q+2}^{2}\left( \Delta \right) \right] +\boldsymbol{\mu }_{x}\boldsymbol{\mu }_{x}^{\top}\mathcal{E}\left[ \varphi \chi
_{q+4}^{2}\left( \Delta \right) \right]
\end{align*}
where $\chi_{v}^{2}\left( \Delta \right)$ is a non-central chi-square distribution with $v$ degrees of freedom and non-centrality parameter $\Delta$.
\end{lemma}

\begin{proof} It can be found in \citet{judge-bock1978} \end{proof}

The bias expressions for the listed estimators are given in the following theorem:
\begin{theorem}
\label{ADB-Low-Linear}
\begin{eqnarray*}
\mathcal{B}\left( \hbbeta_{1}^{\rm LFM}\right) &=&
-\boldsymbol{\mu}_{11.2}\\
\mathcal{B}\left( \hbbeta_{1}^{\rm LSM}\right) &=&-%
\boldsymbol{\gamma }\\
\mathcal{B}\left( \hbbeta_{1}^{\rm LPT}\right) &=&-\boldsymbol{\mu
}_{11.2}-\boldsymbol{\delta }H_{p_{2}+2}\left( \chi _{p_{2},\alpha
}^{2};\Delta \right) ,\text{ } \\
\mathcal{B}\left( \hbbeta_{1}^{\rm LS}\right) &=&-\boldsymbol{\mu }%
_{11.2}-(p_{2}-2)\boldsymbol{\delta }\mathcal{E}\left( \chi _{p_{2}+2
}^{-2}\left( \Delta \right) \right) , \\
\mathcal{B}\left( \hbbeta_{1}^{\rm LPS}\right) &=&-\boldsymbol{\mu
}_{11.2}-\boldsymbol{\delta }H_{p_{2}+2}\left( \chi _{p_{2},\alpha
}^{2};\Delta \right) , \\
&&-(p_{2}-2)\boldsymbol{\delta }\mathcal{E}\left\{ \chi _{p_{2}+2 }^{-2}\left(
\Delta \right) I\left( \chi _{p_{2}+2 }^{2}\left( \Delta \right)
>p_{2}-2\right) \right\} ,
\end{eqnarray*}%
where $\Delta =\left( \boldsymbol{\kappa}^{\top}\bC%
_{22.1}^{-1}\boldsymbol{\kappa}\right) \sigma ^{-2}$, $\bC_{22.1}=%
\bC_{22}-\bC_{21}\bC_{11}^{-1}\bC%
_{12}$, and $H_{v}\left( x,\Delta \right) $ is the cumulative distribution
function of the non-central chi-squared distribution with non-centrality
parameter $\Delta $ and $v$ degree of freedom, and
\begin{equation*}
\mathcal{E}\left( \chi _{v}^{-2j}\left( \Delta \right) \right)
=\int\nolimits_{0}^{\infty }x^{-2j}dH_{v}\left( x,\Delta \right) .
\end{equation*}
\end{theorem}

\begin{proof}

$\mathcal{B}\left( \hbbeta_{1}^{\rm LFM}\right)=-\boldsymbol{\mu }_{11.2}$ is provided by Proposition~\ref{pro1}, and

\begin{eqnarray*}
\mathcal{B}\left( \hbbeta_{1}^{\rm LSM} \right)
&=&\mathcal{E}\left\{ \underset{n\rightarrow \infty }{\lim }\sqrt{n}\left( \boldsymbol{%
\widehat{\beta}}_{1}^{\rm LSM} -\bbeta_{1}\right) \right\} \\
&=&\mathcal{E}\left\{ \underset{n\rightarrow \infty }{\lim }\sqrt{n}\left( \boldsymbol{%
\widehat{\beta}}_{1}^{\rm LFM} +\left(\boldsymbol{C_{11}+I_{p_1}}\right)^{-1}\left( \boldsymbol{C_{11}}+d\boldsymbol{I_{p_1}}\right)\bC_{12}%
\hbbeta_{2}^{\rm LFM} -\bbeta_{1}\right) \right\} \\
&=&\mathcal{E}\left\{ \underset{n\rightarrow \infty }{\lim }\sqrt{n}\left( \boldsymbol{%
\widehat{\beta}}_{1}^{\rm LFM} -\bbeta_{1}\right) \right\} +\mathcal{E}\left\{ \underset{n\rightarrow \infty }{\lim }\sqrt{n}\left(\left(\boldsymbol{C_{11}+I_{p_1}}\right)^{-1} \left( \boldsymbol{C_{11}}+d\boldsymbol{I_{p_1}}\right)\bC_{12}\hbbeta_{2}^{\rm LFM} \right) \right\} \\
&=&-\boldsymbol{\mu }_{11.2}+\bF_d^{11}\bC_{12}\boldsymbol{\kappa} =-\left( \boldsymbol{\mu }_{11.2}-\boldsymbol{\delta }\right)
=-\boldsymbol{\gamma } 
\end{eqnarray*}

Hence, by using Lemma~\ref{lem_JB}, it can be written
as follows:
\begin{eqnarray*}
\mathcal{B}\left( \hbbeta_{1}^{\rm LPT} \right)
&=&\mathcal{E}\left\{ \underset{n\rightarrow \infty }{\lim }\sqrt{n}\left( \boldsymbol{%
\widehat{\beta}}_{1}^{\rm LPT} -\bbeta_{1}\right) \right\} \\
&=&\mathcal{E}\left\{ \underset{n\rightarrow \infty }{\lim }\sqrt{n}\left( \boldsymbol{%
\widehat{\beta}}_{1}^{\rm LFM} -\left( \hbbeta_{1}^{\rm LFM} -%
\hbbeta_{1}^{\rm LSM} \right) I\left( \mathscr{L}_{n}\leq
c_{n,\alpha }\right) -\bbeta_{1}\right) \right\} \\
&=&\mathcal{E}\left\{ \underset{n\rightarrow \infty }{\lim }\sqrt{n}\left( \boldsymbol{%
\widehat{\beta}}_{1}^{\rm LFM} -\bbeta_{1}\right) \right\} \\
&&-\mathcal{E}\left\{ \underset{n\rightarrow \infty }{\lim }\sqrt{n}\left( \left(
\hbbeta_{1}^{\rm LFM} -\hbbeta_{1}^{\rm LSM}
\right) I\left( \mathscr{L}_{n}\leq c_{n,\alpha }\right) \right) \right\} \\
&=&-\boldsymbol{\mu }_{11.2}-\boldsymbol{\delta }H_{p_{2}+2}\left( \chi
_{p_{2},\alpha }^{2};\Delta \right) \text{.}
\end{eqnarray*}

\begin{eqnarray*}
\mathcal{B}\left( \hbbeta_{1}^{\rm LS} \right)
&=&\mathcal{E}\left\{ \underset{n\rightarrow \infty }{\lim }\sqrt{n}\left( \boldsymbol{%
\widehat{\beta}}_{1}^{\rm LS} -\bbeta _{1}\right) \right\} \\
&=&\mathcal{E}\left\{ \underset{n\rightarrow \infty }{\lim }\sqrt{n}\left( \boldsymbol{%
\widehat{\beta}}_{1}^{\rm LFM} -\left( \hbbeta_{1}^{\rm LFM} -%
\hbbeta_{1}^{\rm LSM} \right) \left( p_{2}-2\right) %
\mathscr{L}_{n}^{-1}-\bbeta_{1}\right) \right\} \\
&=&\mathcal{E}\left\{ \underset{n\rightarrow \infty }{\lim }\sqrt{n}\left( \boldsymbol{%
\widehat{\beta}}_{1}^{\rm LFM} -\bbeta_{1}\right) \right\}\\
&& -\mathcal{E}\left\{ \underset{n\rightarrow \infty }{\lim }\sqrt{n}\left( \left(
\hbbeta_{1}^{\rm LFM} -\hbbeta_{1}^{\rm LSM}
\right) \left( p_{2}-2\right) \mathscr{L}_{n}^{-1}\right) \right\} \\
&=&-\boldsymbol{\mu }_{11.2}-\left( p_{2}-2\right) \boldsymbol{\delta }%
\mathcal{E}\left( \chi _{p_{2}+2}^{-2}\left( \Delta \right) \right) \text{.}
\end{eqnarray*}
\begin{eqnarray*}
\mathcal{B}\left( \hbbeta_{1}^{\rm LPS} \right)
&=&\mathcal{E}\left\{ \underset{n\rightarrow \infty }{\lim }\sqrt{n}\left( \boldsymbol{%
\widehat{\beta}}_{1}^{\rm LPS} -\bbeta_{1}\right) \right\} \\
&=&\mathcal{E}\left\{ \underset{n\rightarrow \infty }{\lim }\sqrt{n}( \boldsymbol{%
\widehat{\beta}}_{1}^{\rm LSM} +\left( \hbbeta_{1}^{\rm LFM} -\boldsymbol{%
\widehat{\beta}}_{1}^{\rm LSM} \right) \left( 1-\left( p_{2}-2\right) \mathscr{L}%
_{n}^{-1}\right) \right. \\
&&\left. I\left( \mathscr{L}_{n}>p_{2}-2\right) -\bbeta_{1}) \right\} \\
&=&\mathcal{E}\left\{ \underset{n\rightarrow \infty }{\lim }\sqrt{n}\left[ \boldsymbol{%
\widehat{\beta}}_{1}^{\rm LSM} +\left( \hbbeta_{1}^{\rm LFM} -\boldsymbol{%
\widehat{\beta}}_{1}^{\rm LSM} \right) \left( 1-I\left( \mathscr{L}_{n}\leq
p_{2}-2\right) \right) \right. \right. \\
&&\left. \left. -\left( \hbbeta_{1}^{\rm LFM}
 -\hbbeta_{1}^{\rm LSM} \right) \left(
p_{2}-2\right) \mathscr{L}_{n}^{-1}I\left( \mathscr{L}_{n}>p_{2}-2\right)
-\bbeta_{1}\right] \right\} \\
&=&\mathcal{E}\left\{ \underset{n\rightarrow \infty }{\lim }\sqrt{n}\left( \boldsymbol{%
\widehat{\beta}}_{1}^{\rm LFM} -\bbeta_{1}\right) \right\} \\
&&-\mathcal{E}\left\{ \underset{n\rightarrow \infty }{\lim }\sqrt{n}\left( \boldsymbol{%
\widehat{\beta}}_{1}^{\rm LFM} -\hbbeta_{1}^{\rm RSM} \right)
I\left( \mathscr{L}_{n}\leq p_{2}-2\right) \right\} \\
&&-\mathcal{E}\left\{ \underset{n\rightarrow \infty }{\lim }\sqrt{n}\left( \boldsymbol{%
\widehat{\beta}}_{1}^{\rm RFM} -\hbbeta_{1}^{\rm LSM} \right) \left(
p_{2}-2\right) \mathscr{L}_{n}^{-1}I\left( \mathscr{L}_{n}>p_{2}-2\right)
\right\} \\
&=&-\boldsymbol{\mu }_{11.2}-\boldsymbol{\delta }H_{p_{2}+2}\left(p_{2}-2;\left( \Delta \right) \right) \\
&&-\boldsymbol{\delta }\left( p_{2}-2\right) \mathcal{E}\left\{ \chi
_{p_{2}+2}^{-2}\left( \Delta \right) I\left(\chi _{p_{2+2}}^{2}\left( \Delta
\right) >p_{2}-2\right)\right\} \text{.}	\qedhere		
\end{eqnarray*}

\end{proof}

Now, since we defined the asymptotic distributional quadratic bias of an estimator $\hbbeta_{1}^{\rm *}$ as follows
\begin{eqnarray*}
\mathcal{QB}\left( \hbbeta_{1}^{\rm *}\right)=\left( \mathcal{B}\left( \hbbeta_{1}^{\rm *}\right) \right)' \bS_{11.2}\left( \mathcal{B}\left( \hbbeta_{1}^{\rm *}\right) \right).
\end{eqnarray*}
The following asymptotic distributional biases of the estimators $\hbbeta_{1}^{\rm LFM}$, $\hbbeta_{1}^{\rm LSM}$, $\hbbeta_{1}^{\rm LPT}$, $\hbbeta_{1}^{\rm LS}$, and $\hbbeta_{1}^{\rm LPS}$ are obtained respectively,

\begin{eqnarray*}
\mathcal{QB}\left( \hbbeta_{1}^{\rm LFM}\right) &=&\boldsymbol{\mu
}_{11.2}^{\top}\bS_{11.2}\boldsymbol{\mu }_{11.2}, \\
\mathcal{QB}\left( \hbbeta_{1}^{\rm LSM}\right) &=&\boldsymbol{\gamma}^{\top}\bS_{11.2}\boldsymbol{\gamma},\\
\mathcal{QB}\left( \hbbeta_{1}^{\rm LPT}\right) &=&\boldsymbol{\mu
}_{11.2}^{\top}\bS_{11.2}\boldsymbol{\mu }_{11.2}+\boldsymbol{%
\mu }_{11.2}^{\top}\bS_{11.2}\boldsymbol{\delta }%
H_{p_{2}+2}\left( \chi _{p_{2},\alpha }^{2};\Delta \right) \\
&&+\boldsymbol{\delta }^{\top}\bS_{11.2}\boldsymbol{\mu }%
_{11.2}H_{p_{2}+2}\left( \chi _{p_{2},\alpha }^{2};\Delta \right) \\
&&+\boldsymbol{\delta }^{\top}\bS_{11.2}\boldsymbol{\delta }%
H_{p_{2}+2}^{2}\left( \chi _{p_{2},\alpha }^{2};\Delta \right), \\
\mathcal{QB}\left( \hbbeta_{1}^{\rm LS}\right) &=&\boldsymbol{\mu }%
_{11.2}^{\top}\bS_{11.2}\boldsymbol{\mu }_{11.2}+(p_{2}-2)%
\boldsymbol{\mu }_{11.2}^{\top}\bS_{11.2}\boldsymbol{\delta }%
\mathcal{E}\left( \chi _{p_{2}+2 }^{-2}\left( \Delta \right) \right) \\
&&+(p_{2}-2)\boldsymbol{\delta }^{\top}\bS_{11.2}\boldsymbol{\mu
}_{11.2}\mathcal{E}\left( \chi _{p_{2}+2 }^{-2}\left( \Delta \right) \right) \\
&&+(p_{2}-2)^{2}\boldsymbol{\delta }^{\top}\bS_{11.2}\boldsymbol{%
\delta }\left( \mathcal{E}\left( \chi _{p_{2}+2 }^{-2}\left( \Delta \right)
\right) \right) ^{2}, \\
\mathcal{QB}\left( \hbbeta_{1}^{\rm LPS}\right) &=&\boldsymbol{\mu
}_{11.2}^{\top}\bS_{11.2}\boldsymbol{\mu }_{11.2}+\left(
\boldsymbol{\delta }^{\top}\bS_{11.2}\boldsymbol{\mu }_{11.2}+%
\boldsymbol{\mu }_{11.2}^{\top}\bS_{11.2}\boldsymbol{\delta }%
\right) \\
&&\cdot \left[ H_{p_{2}+2}\left(p_{2}-2;\Delta \right) \right. \\
&&\left. +(p_{2}-2)\mathcal{E}\left\{ \chi _{p_{2}+2 }^{-2}\left( \Delta
\right) I\left( \chi _{p_{2}+2 }^{-2}\left( \Delta \right)
>p_{2}-2\right) \right\} \right] \\
&&+\boldsymbol{\delta }^{\top}\bS_{11.2}\boldsymbol{\delta }%
\left[ H_{p_{2}+2}\left(p_{2}-2;\Delta \right) \right. \\
&&\left. +(p_{2}-2)\mathcal{E}\left\{ \chi _{p_{2}+2 }^{-2}\left( \Delta
\right) I\left( \chi _{p_{2}+2 }^{-2}\left( \Delta \right)
>p_{2}-2\right) \right\} \right] ^{2}.
\end{eqnarray*}

Here we get similar pattern with the results of \cite{yuzbasi-et-al2017}.  So, we omit the details here.

In order to compute the risk functions, we firstly, compute the asymptotic covariance of the estimators. The asymptotic covariance of an estimator $\hbbeta_{1}^{\rm *}$ is obtained by
\begin{eqnarray*}
\boldsymbol{\Gamma} \left( \hbbeta_{1}^{\rm *} \right) &=&\mathcal{E}\left\{
\underset{n\rightarrow \infty }{\lim }\sqrt{n}\left( \boldsymbol{\widehat{%
\beta}}_{1}^{\rm *} -\bbeta_{1}\right) \sqrt{n}\left( \boldsymbol{\widehat{%
\beta}}_{1}^{\rm *} -\bbeta_{1}\right) ^{\top }\right\}. \\
\end{eqnarray*}

Now, we simply start by computing the asymptotic covariance of the estimator $\hbbeta_{1}^{\rm LFM}$ as follows:
\begin{eqnarray}\label{ACov_LFM}
\boldsymbol{\Gamma} \left( \hbbeta_{1}^{\rm LFM} \right) =\sigma ^{2}\bS_{11.2}^{-1}+\boldsymbol{\mu}_{11.2}\boldsymbol{%
\mu}_{11.2}^{\top }.
\end{eqnarray}
Similarly, the asymptotic covariance of the estimator $\hbbeta_{1}^{\rm LSM}$ is obtained as
\begin{eqnarray}\label{ACov_LSM}
\boldsymbol{\Gamma} \left( \hbbeta_{1}^{\rm LFM} \right) =\sigma ^{2}\bS_{11.2}^{-1}+\boldsymbol{\gamma}_{11.2}\boldsymbol{%
%
%
\gamma}_{11.2}^{\top }.
\end{eqnarray}

The asymptotic covariance of the estimator $\hbbeta_{1}^{\rm LPT}$ can be obtained by computing the following
\begin{eqnarray*}
\boldsymbol{\Gamma} \left( \hbbeta_{1}^{\rm LPT} \right)
&=&\mathcal{E}\left\{ \underset{n\rightarrow \infty }{\lim }\sqrt{n}\left( \boldsymbol{%
\widehat{\beta}}_{1}^{\rm LPT} -\bbeta_{1}\right) \sqrt{n}\left( \boldsymbol{%
\widehat{\beta}}_{1}^{\rm LPT} -\bbeta_{1}\right) ^{\top }\right\} \nonumber \\
&=&\mathcal{E}\left\{ \underset{n\rightarrow \infty }{\lim }n\left[ \left( \boldsymbol{%
\widehat{\beta}}_{1}^{\rm LFM} -\bbeta_{1}\right) -\left( \boldsymbol{\widehat{%
\beta}}_{1}^{\rm LFM} -\hbbeta_{1}^{\rm LSM} \right) I\left( %
\mathscr{L}_{n}\leq c_{n,\alpha }\right) \right] \right. \nonumber\\
&&\left. \left[ \left( \hbbeta_{1}^{\rm LFM} -\bbeta_{1}\right) -\left( \hbbeta_{1}^{\rm LFM} -\boldsymbol{%
\widehat{\beta}}_{1}^{\rm LSM} \right) I\left( \mathscr{L}_{n}\leq c_{n,\alpha
}\right) \right] ^{\top }\right\} \nonumber\\
&=&\mathcal{E}\left\{ \left[ \vartheta _{1}-\vartheta _{3}I\left( \mathscr{L}_{n}\leq
c_{n,\alpha }\right) \right] \left[ \vartheta _{1}-\vartheta _{3}I\left( %
\mathscr{L}_{n}\leq c_{n,\alpha }\right) \right] ^{\top }\right\} \nonumber\\
&=&\mathcal{E}\left\{ \vartheta _{1}\vartheta _{1}^{\top }-2\vartheta _{3}\vartheta
_{1}^{\top }I\left( \mathscr{L}_{n}\leq c_{n,\alpha }\right) +\vartheta
_{3}\vartheta _{3}^{\top }I\left( \mathscr{L}_{n}\leq c_{n,\alpha }\right)
\right\} \text{.}
\end{eqnarray*}

Thus, we need to compute $\mathcal{E}\left\{ \vartheta _{1}\vartheta _{1}^{\top }
\right\}$, $\mathcal{E}\left\{\vartheta _{3}\vartheta
_{1}^{\top }I\left( \mathscr{L}_{n}\leq c_{n,\alpha }\right)\right\}$ and $\mathcal{E}\left\{ \vartheta
_{3}\vartheta _{3}^{\top }I\left( \mathscr{L}_{n}\leq c_{n,\alpha }\right)\right\}$.
Since the first term is $\mathcal{E}\left\{ \vartheta _{1}\vartheta _{1}^{\top }
\right\}=\sigma ^{2}\bS_{11.2}^{-1}+\boldsymbol{\mu}_{11.2}\boldsymbol{%
\mu}_{11.2}^{\top }$, by using Lemma (\ref{lem_JB}), we compute the third term as
$$ \mathcal{E}\left\{ \vartheta _{3}\vartheta _{3}^{\top }I\left( \mathscr{L}_{n}\leq
c_{n,\alpha }\right) \right\}= \boldsymbol{\Phi }H_{p_{2}+2}\left( \chi _{p_{2},\alpha }^{2};\Delta \right)+\boldsymbol{\delta \delta }^{\top }H_{p_{2}+4}\left( \chi _{p_{2},\alpha
}^{2};\Delta \right).$$

Finally, we use the formula of a conditional mean of a bivariate normal distribution and obtain 
\begin{eqnarray*}
\lefteqn{\mathcal{E}\left\{ \vartheta _{3}\vartheta _{1}^{\top }I\left( \mathscr{L}_{n}\leq
c_{n,\alpha }\right) \right\}}\\
&=&\mathcal{E}\left\{ \mathcal{E}\left( \vartheta _{3}\vartheta _{1}^{\top }I\left( \mathscr{L}%
_{n}\leq c_{n,\alpha }\right) |\vartheta _{3}\right) \right\}
=\mathcal{E}\left\{ \vartheta _{3}\mathcal{E}\left( \vartheta _{1}^{\top }I\left( \mathscr{L}%
_{n}\leq c_{n,\alpha }\right) |\vartheta _{3}\right) \right\} \\
&=&\mathcal{E}\left\{ \vartheta _{3}\left[ -\mu _{11.2}+\left( \vartheta _{3}-%
\boldsymbol{\delta }\right) \right] ^{\top }I\left( \mathscr{L}_{n}\leq
c_{n,\alpha }\right) \right\} \\
&=&-\mathcal{E}\left\{ \vartheta _{3}\boldsymbol{\mu }_{11.2}^{\top }I\left( %
\mathscr{L}_{n}\leq c_{n,\alpha }\right) \right\}
+\mathcal{E}\left\{ \vartheta _{3}\left( \vartheta _{3}-\boldsymbol{\delta }\right)
^{\top }I\left( \mathscr{L}_{n}\leq c_{n,\alpha }\right) \right\} \\
&=&-\boldsymbol{\mu }_{11.2}^{\top }\mathcal{E}\left\{ \vartheta _{3}I\left( %
\mathscr{L}_{n}\leq c_{n,\alpha }\right) \right\}
+\mathcal{E}\left\{ \vartheta _{3}\vartheta _{3}^{\top }I\left( \mathscr{L}_{n}\leq
c_{n,\alpha }\right) \right\}
-\mathcal{E}\left\{ \vartheta _{3}\boldsymbol{\delta }^{\top }I\left( \mathscr{L}%
_{n}\leq c_{n,\alpha }\right) \right\}.
\end{eqnarray*}

Now, putting all the terms together and after some easy algebra, we obtain
\begin{eqnarray}\label{ACov_LPT}
\boldsymbol{\Gamma} \left( \hbbeta_{1}^{\rm LPT} \right)
&=& \sigma ^{2}\bS_{11.2}^{-1}+\boldsymbol{\mu }_{11.2}\boldsymbol{%
\mu }_{11.2}^{\top }+2\boldsymbol{\mu }_{11.2}^{\top }\boldsymbol{\delta }%
H_{p_{2}+2}\left( \chi _{p_{2},\alpha }^{2};\Delta \right) \nonumber\\
&&-\boldsymbol{\Phi }H_{p_{2}+2}\left( \chi _{p_{2},\alpha }^{2};
\Delta  \right) +\boldsymbol{\delta \delta }^{\top }\left[ 2H_{p_{2}+2}\left( \chi
_{p_{2},\alpha }^{2};\Delta \right) -H_{p_{2}+4}\left( \chi _{p_{2},\alpha
}^{2};\Delta \right) \right].
\end{eqnarray}

The asymptotic covariance of $\hbbeta_{1}^{\rm LS}$ can be obtained by
\begin{eqnarray*}
\boldsymbol{\Gamma} \left( \hbbeta_{1}^{\rm LS}\right)
&=&\mathcal{E}\left\{ \underset{n\rightarrow \infty }{\lim }\sqrt{n}\left( \boldsymbol{%
\widehat{\beta }}_{1}^{\rm LS}-\bbeta_{1}\right) \sqrt{n}\left(
\hbbeta_{1}^{\rm LS}-\bbeta_{1}\right) ^{\top
}\right\}  \\
&=&\mathcal{E}\left\{ \underset{n\rightarrow \infty }{\lim }n\left[ \left( \boldsymbol{%
\widehat{\beta }}_{1}^{\rm LFM}-\bbeta_{1}\right) -\left(
\hbbeta_{1}^{\rm LFM}-\hbbeta%
_{1}^{\rm LSM}\right) \left( p_{2}-2\right) \mathscr{L}_{n}^{-1}\right] \right.
\\
&&\left. \left[ \left( \hbbeta_{1}^{\rm LFM}-\boldsymbol{%
\beta }_{1}\right) -\left( \hbbeta_{1}^{\rm LFM}-\boldsymbol{%
\widehat{\beta }}_{1}^{\rm LSM}\right) \left( p_{2}-2\right) \mathscr{L}_{n}^{-1}%
\right] ^{\top }\right\}  \\
&=&\mathcal{E}\left\{ \vartheta _{1}\vartheta _{1}^{\top }-2\left( p_{2}-2\right)
\vartheta _{3}\vartheta _{1}^{\top }\mathscr{L}_{n}^{-1}+\left(
p_{2}-2\right) ^{2}\vartheta _{3}\vartheta _{3}^{\top }\mathscr{L}%
_{n}^{-2}\right\} \text{.}
\end{eqnarray*}%
Thus, we need to compute $\mathcal{E}\left\{\vartheta _{3}\vartheta _{3}^{\top } \mathscr{L}%
_{n}^{-2}\right\}$ and $\mathcal{E}\left\{\vartheta _{3}\vartheta _{1}^{\top }\mathscr{L}%
_{n}^{-1} \right\}$.
By using Lemma (\ref{lem_JB}), the first one is obtained as follows:
$$\mathcal{E}\left\{\vartheta _{3}\vartheta _{3}^{\top } \mathscr{L}%
_{n}^{-2}\right\}=\boldsymbol{\Phi }_{
}\mathcal{E}\left( \chi _{p_{2}+2}^{-4}\left( \Delta \right) \right) +\boldsymbol{\delta \delta }^{\top }\mathcal{E}\left( \chi _{p_{2}+4}^{-4}\left(
\Delta \right) \right).$$
To compute the second one, we again need the formula of a conditional mean of a bivariate normal distribution and get
\begin{eqnarray*}
\mathcal{E}\left\{ \vartheta _{3}\vartheta _{1}^{\top }\mathscr{L}_{n}^{-1}\right\}
&=&\mathcal{E}\left\{ \mathcal{E}\left( \vartheta _{3}\vartheta _{1}^{\top }\mathscr{L}%
_{n}^{-1}|\vartheta _{3}\right) \right\}
=\mathcal{E}\left\{ \vartheta _{3}\mathcal{E}\left( \vartheta _{1}^{\top }\mathscr{L}%
_{n}^{-1}|\vartheta _{3}\right) \right\}  \\
&=&\mathcal{E}\left\{ \vartheta _{3}\left[ -\boldsymbol{\mu }_{11.2}+\left( \vartheta
_{3}-\boldsymbol{\delta }\right) \right] ^{\top }\mathscr{L}%
_{n}^{-1}\right\} \\
&=&-\mathcal{E}\left\{ \vartheta _{3}\boldsymbol{\mu }_{11.2}^{\top }\mathscr{L}%
_{n}^{-1}\right\} +\mathcal{E}\left\{ \vartheta _{3}\left( \vartheta _{3}-\boldsymbol{%
\delta }\right) ^{\top }\mathscr{L}_{n}^{-1}\right\} \\
&=&-\boldsymbol{\mu }_{11.2}^{\top }\mathcal{E}\left\{ \vartheta _{3}\mathscr{L}%
_{n}^{-1}\right\} +\mathcal{E}\left\{ \vartheta _{3}\vartheta _{3}^{\top }\mathscr{L}%
_{n}^{-1}\right\}
-\mathcal{E}\left\{ \vartheta _{3}\boldsymbol{\delta }^{\top }\mathscr{L}%
_{n}^{-1}\right\}.
\end{eqnarray*}%
We also have $\mathcal{E}\left\{ \vartheta _{3}\boldsymbol{\delta }^{\top }\mathscr{L}%
_{n}^{-1}\right\}= \boldsymbol{\delta \delta}^{\top}\mathcal{E}\left( \chi _{p_{2}+2}^{-2}\left( \Delta \right) \right)$  
and $\mathcal{E}\left\{ \vartheta _{3}\mathscr{L}_{n}^{-1}\right\}= \boldsymbol{\delta}\mathcal{E}\left( \chi _{p_{2}+2}^{-2}\left( \Delta \right) \right)$.
Therefore, after some algebra we get
\begin{eqnarray}\label{ACov_LS}
\lefteqn{\boldsymbol{\Gamma} \left( \hbbeta_{1}^{\rm LS}\right)}\\
&=&\sigma ^{2}\bS_{11.2}^{-1}+\boldsymbol{\mu }_{11.2}\boldsymbol{%
\mu }_{11.2}^{\top }+2\left( p_{2}-2\right) \boldsymbol{\mu }_{11.2}^{\top }%
\boldsymbol{\delta }\mathcal{E}\left( \chi _{p_{2+2}}^{-2}\left( \Delta
\right) \right) \nonumber \\
&&-\left( p_{2}-2\right) \boldsymbol{\Phi }\left\{ 2\mathcal{E}\left( \chi
_{p_{2}+2}^{-2}\left( \Delta \right) \right) -\left( p_{2}-2\right) \mathcal{E}\left(
\chi _{p_{2+2}}^{-4}\left( \Delta \right) \right) \right\} \nonumber\\
&&+\left( p_{2}-2\right) \boldsymbol{\delta \delta }^{\top }\left\{
-2\mathcal{E}\left( \chi _{p_{2+4}}^{-2}\left( \Delta \right) \right) +2\mathcal{E}\left( \chi
_{p_{2}+2}^{-2}\left( \Delta \right) \right) +\left( p_{2}-2\right) \mathcal{E}\left(
\chi _{p_{2+4}}^{-4}\left( \Delta \right) \right) \right\} \text{.}
\end{eqnarray}

Finally, we compute the asymptotic covariance of $\boldsymbol{\Gamma} \left( \hbbeta_{1}^{\rm LPS}\right)$ as follows:

\begin{eqnarray*}
\boldsymbol{\Gamma} \left( \hbbeta_{1}^{\rm LPS} \right)
&=&\mathcal{E}\left\{ \underset{n\rightarrow \infty }{\lim }n\left( \hbbeta_{1}^{\rm LPS} -\bbeta_{1}\right) \left( \hbbeta_{1}^{\rm LPS}
-\bbeta_{1}\right) ^{\top}\right\} \\
&=&\mathcal{E}\left\{ \underset{n\rightarrow \infty }{\lim }n\left( \hbbeta_{1}^{\rm LS}-\bbeta_{1}\right)\left( \hbbeta_{1}^{\rm LS}-\bbeta_{1}\right)^{\top} \right\} \\
&&- 2\mathcal{E}\left\{\underset{n\rightarrow \infty }{\lim }n \left( \hbbeta_{1}^{\rm LFM}-\hbbeta_{1}^{\rm LSM}\right)\left( \hbbeta_{1}^{\rm LS}-\hbbeta_{1}\right)^{\top}\left\{ 1-\left( p_{2}-2\right) \mathscr{L}%
_{n}^{-1}\right\}I\left( \mathscr{L}_{n}\leq p_{2}-2\right) \right\}\\
&&+\mathcal{E}\left\{\underset{n\rightarrow \infty }{\lim }n \left( \hbbeta_{1}^{\rm LFM}-\hbbeta_{1}^{\rm LSM}\right)\left( \hbbeta_{1}^{\rm LFM}-\hbbeta_{1}^{\rm LSM}\right)^{\top}\left\{ 1-\left( p_{2}-2\right) \mathscr{L}%
_{n}^{-1}\right\}^2 I\left( \mathscr{L}_{n}\leq p_{2}-2\right)\right\}\\
&=& \boldsymbol{\Gamma} \left( \hbbeta_{1}^{\rm LS} \right)-2\mathcal{E}\left\{ \vartheta _{3}\vartheta _{1}^{\top }\left[1-\left( p_{2}-2\right)\mathscr{L}_{n}^{-1}\right]I\left( \mathscr{L}_{n}\leq p_{2}-2\right)\right\}\\
&& +2\mathcal{E}\left\{ \vartheta _{3}\vartheta _{3}^{\top }\left( p_{2}-2\right)\mathscr{L}_{n}^{-1}\left(1-\left( p_{2}-2\right)\mathscr{L}_{n}^{-1}\right) I\left( \mathscr{L}_{n}\leq p_{2}-2\right)\right\}\\
&&+\mathcal{E}\left\{\vartheta _{3}\vartheta _{3}^{\top }\left(1-\left( p_{2}-2\right)\mathscr{L}_{n}^{-1}\right)^2 I\left( \mathscr{L}_{n}\leq p_{2}-2\right)\right\}.
\end{eqnarray*}
Simplifying the equation above, we get
\begin{eqnarray*}
\boldsymbol{\Gamma} \left( \hbbeta_{1}^{\rm LPS} \right)&=& \boldsymbol{\Gamma} \left( \hbbeta_{1}^{\rm LS} \right)-2\mathcal{E}\left\{ \vartheta _{3}\vartheta _{1}^{\top }\left[1-\left( p_{2}-2\right)\mathscr{L}_{n}^{-1}\right]I\left( \mathscr{L}_{n}\leq p_{2}-2\right)\right\}\\
&&-\mathcal{E}\left\{ \vartheta _{3}\vartheta _{3}^{\top }\left( p_{2}-2\right)^2\mathscr{L}_{n}^{-2} I\left( \mathscr{L}_{n}\leq p_{2}-2\right)\right\}+\mathcal{E}\left\{ \vartheta _{3}\vartheta _{3}^{\top }I\left( \mathscr{L}_{n}\leq p_{2}-2\right)\right\}.\\
\end{eqnarray*}
Now, we need to compute the expectations obtained in the above equation. We firstly compute the last one as
$$\mathcal{E}\left\{ \vartheta _{3}\vartheta _{3}^{\top }I\left( \mathscr{L}_{n}\leq p_{2}-2\right)\right\}= \boldsymbol{\Phi }H_{p_{2}+2}\left( p_2-2;\Delta \right)+\boldsymbol{\delta \delta }^{\top }H_{p_{2}+4}\left( p_2-2;\Delta \right)$$

By using Lemma(\ref{lem_JB}) and using the formula of a conditional mean of a bivariate normal distribution, the first expectation becomes
\begin{eqnarray*}
\lefteqn{\mathcal{E}\left\{ \vartheta _{3}\vartheta _{1}^{\top }\left[1-\left( p_{1}-2\right)\mathscr{L}_{n}^{-1}\right]I\left( \mathscr{L}_{n}\leq p_{2}-2\right)\right\}}\\
=&-& \boldsymbol{\delta \mu }_{11.2}^{\top }\mathcal{E}\left( \left\{
1-\left( p_{2}-2\right) \chi _{p_{2}+2}^{-2}\left( \Delta \right) \right\}
I\left( \chi _{p_{2}+2}^{2}\left( \Delta \right) \leq p_{2}-2\right) \right)
\\
&+&\boldsymbol{\Phi }\mathcal{E}\left( \left\{ 1-\left( p_{2}-2\right) \chi
_{p_{2}+2}^{-2}\left( \Delta \right) \right\} I\left( \chi
_{p_{2}+2}^{2}\left( \Delta \right) \leq p_{2}-2\right) \right) \\
&+&\boldsymbol{\delta \delta }^{\top }\mathcal{E}\left( \left\{ 1-\left(
p_{2}-2\right) \chi _{p_{2}+4}^{-2}\left( \Delta \right) \right\}
I\left( \chi _{p_{2}+4}^{2}\left( \Delta \right) \leq p_{2}-2\right) \right)
\\
&-&\boldsymbol{\delta \delta }^{\top }\mathcal{E}\left( \left\{ 1-\left(
p_{2}-2\right) \chi _{p_{2}+2}^{-2}\left( \Delta \right) \right\}
I\left( \chi _{p_{2}+2}^{2}\left( \Delta \right) \leq p_{2}-2\right) \right).
\end{eqnarray*} 
Thus, the asymptotic covariance of $\boldsymbol{\Gamma}\left(\hbbeta_{1}^{\rm LPS}\right)$ can be written as follows:

\begin{eqnarray}\label{ACov_PLS}
\boldsymbol{\Gamma} \left( \hbbeta_{1}^{\rm LPS} \right)
&=&\boldsymbol{\Gamma} \left( \hbbeta_{1}^{\rm RS} \right) +2\boldsymbol{%
\delta \mu }_{11.2}^{\top }\mathcal{E}\left( \left\{ 1-\left( p_{2}-2\right) \chi
_{p_{2}+2}^{-2}\left( \Delta \right) \right\} I\left( \chi
_{p_{2}+2}^{2}\left( \Delta \right) \leq  p_{2}-2\right)
\right)\nonumber \\
&&-2\boldsymbol{\Phi }\mathcal{E}\left( \left\{ 1-\left( p_{2}-2\right) \chi
_{p_{2}+2}^{-2}\left( \Delta \right) \right\} I\left( \chi
_{p_{2}+2}^{-2}\left( \Delta \right) \leq p_{2}-2\right) \right) \nonumber \\
&&-2\boldsymbol{\delta \delta }^{\top }\mathcal{E}\left( \left\{ 1-\left(
p_{2}-2\right) \chi _{p_{2}+4}^{-2}\left( \Delta \right) \right\} I\left(
\chi _{p_{2}+4}^{2}\left( \Delta \right) \leq p_{2}-2\right) \right) \nonumber \\
&&+2\boldsymbol{\delta \delta }^{\top }\mathcal{E}\left( \left\{ 1-\left(
p_{2}-2\right) \chi _{p_{2}+2}^{-2}\left( \Delta \right) \right\} I\left(
\chi _{p_{2}+2}^{2}\left( \Delta \right) \leq p_{2}-2\right) \right)  \\
&&-\left( p_{2}-2\right) ^{2}\boldsymbol{\Phi }\mathcal{E}\left( \chi
_{p_{2}+2,\alpha }^{-4}\left( \Delta \right) I\left( \chi _{p_{2}+2,\alpha
}^{2}\left( \Delta \right) \leq p_{2}-2\right) \right) \nonumber \\
&&-\left( p_{2}-2\right) ^{2}\boldsymbol{\delta \delta }^{\top }\mathcal{E}\left( \chi
_{p_{2}+4}^{-4}\left( \Delta \right) I\left( \chi _{p_{2}+2}^{2}\left(
\Delta \right) \leq p_{2}-2\right) \right) \nonumber \\
&&+\boldsymbol{\Phi }H_{p_{2}+2}\left( p_{2}-2;\Delta \right) +%
\boldsymbol{\delta \delta }^{\top }H_{p_{2}+4}\left( p_{2}-2;\Delta \right).\nonumber
\end{eqnarray}

Based on the computations regarding the asymptotic covariances and using the equations (\ref{ACov_LFM}), (\ref{ACov_LSM}), (\ref{ACov_LPT}), (\ref{ACov_LS}) and (\ref{ACov_PLS}), we present the risks of the estimators $\hbbeta_{1}^{\rm LFM}$, $\hbbeta_{1}^{\rm LSM}$, $\hbbeta_{1}^{\rm LPT}$, $\hbbeta_{1}^{\rm LS}$ and $\hbbeta_{1}^{\rm LPS}$ respectively in the following theorem.
\begin{theorem} Under the local alternatives $\left\{ K_{n}\right\} $ and assuming the regularity conditions (i) and (ii), the risks of the estimators are:
\label{ADR-Low-Linear}
\begin{eqnarray*}
\mathcal{R}\left( \hbbeta_{1}^{\rm LFM}\right)&=&
\sigma^{2}tr\left( \boldsymbol{W}\bS_{11.2}^{-1}\right) +\boldsymbol{\mu }%
_{11.2}^{\top}\boldsymbol{W\mu }_{11.2}\\
\mathcal{R}\left( \hbbeta_{1}^{\rm LSM}\right)&=&
\sigma ^{2}\textnormal{tr}\left( \boldsymbol{W}\bS_{11}^{-1}\right) +\boldsymbol{\gamma}%
^{\top}\boldsymbol{W\gamma}\\
\mathcal{R}\left( \hbbeta_{1}^{\rm LPT}\right)  &=&\mathcal{R}\left(
\hbbeta_{1}^{\rm LFM}\right) +2\boldsymbol{\mu }%
_{11.2}^{\top}\boldsymbol{W\delta }H_{p_{2}+2}\left( \chi _{p_{2},\alpha
}^{2};\Delta \right)-\textnormal{tr}(\boldsymbol{W\Phi})H_{p_{2}+2}\left( \chi _{p_{2},\alpha }^{2};\Delta\right) \\
&&+\boldsymbol{\delta }^{\top}\boldsymbol{W\delta }\left\{
2H_{p_{2}+2}\left( \chi _{p_{2},\alpha }^{2};\Delta \right)
-H_{p_{2}+4}\left( \chi _{p_{2},\alpha }^{2};\Delta \right) \right\},\\
\mathcal{R}\left( \hbbeta_{1}^{\rm LS}\right)  &=&\mathcal{R}\left(
\hbbeta_{1}^{\rm LFM}\right) +2(p_{2}-2)\boldsymbol{\mu }%
_{11.2}^{\top}\boldsymbol{W\delta }\mathcal{E}\left( \chi _{p_{2}+2
}^{-2}\left( \Delta \right) \right)  \\
&&-(p_{2}-2)\textnormal{tr}\left(\boldsymbol{W\Phi}\right) \left[ \mathcal{E}\left( \chi _{p_{2}+2}^{-2}\left( \Delta \right) \right) -(p_{2}-2)\mathcal{E}\left( \chi _{p_{2}+2 }^{-4}\left( \Delta \right)
\right) \right] \\
&&+(p_{2}-2)\boldsymbol{\delta }^{\top}\boldsymbol{W\delta }\left[2\mathcal{E}\left( \chi
_{p_{2}+2}^{-2}\left( \Delta \right) \right) -2\mathcal{E}\left( \chi _{p_{2}+4 }^{-2}\left( \Delta \right) \right)
+(p_{2}-2)\mathcal{E}\left( \chi _{p_{2}+4 }^{-4}\left( \Delta \right) \right)
\right],\\
\mathcal{R}\left( \hbbeta_{1}^{\rm LPS}\right)  &=&\mathcal{R}\left(
\hbbeta_{1}^{\rm RS}\right)+2\boldsymbol{\mu }_{11.2}^{\top}\boldsymbol{W\delta }\mathcal{E}\left( \left\{1-(p_{2}-2)\chi _{p_{2}+2 }^{-2}\left( \Delta \right) \right\}
I\left( \chi _{p_{2}+2 }^{2}\left( \Delta \right) \leq
p_{2}-2\right) \right)  \\
&&+\textnormal{tr}\left(\boldsymbol{W\Phi}\right)\mathcal{E}\left( \left\{ 1-\left( p_{2}-2\right) \chi
_{p_{2}+2}^{-2}\left( \Delta \right) \right\} I\left( \chi
_{p_{2}+2}^{-2}\left( \Delta \right) \leq p_{2}-2\right) \right)\\
&&-2\boldsymbol{\delta^{\top }W \delta }\mathcal{E}\left( \left\{ 1-\left(
p_{2}-2\right) \chi _{p_{2}+4}^{-2}\left( \Delta \right) \right\} I\left(
\chi _{p_{2}+4}^{2}\left( \Delta \right) \leq p_{2}-2\right) \right)  \\
&&+2\boldsymbol{\delta^{\top }W \delta }\mathcal{E}\left( \left\{ 1-\left(
p_{2}-2\right) \chi _{p_{2}+2}^{-2}\left( \Delta \right) \right\} I\left(
\chi _{p_{2}+2}^{2}\left( \Delta \right) \leq p_{2}-2\right) \right)  \\
&&-\left( p_{2}-2\right) ^{2}\textnormal{tr}\left(\boldsymbol{W\Phi}\right)\mathcal{E}\left( \chi
_{p_{2}+2,\alpha }^{-4}\left( \Delta \right) I\left( \chi _{p_{2}+2,\alpha
}^{2}\left( \Delta \right) \leq p_{2}-2\right) \right) \nonumber \\
&&-\left( p_{2}-2\right) ^{2}\boldsymbol{\delta^{\top }W \delta }\mathcal{E}\left( \chi
_{p_{2}+4}^{-4}\left( \Delta \right) I\left( \chi _{p_{2}+2}^{2}\left(
\Delta \right) \leq p_{2}-2\right) \right) \nonumber \\
&&+\textnormal{tr}\left(\boldsymbol{W\Phi}\right)H_{p_{2}+2}\left( p_{2}-2;\Delta \right) +%
\boldsymbol{\delta^{\top }W \delta }H_{p_{2}+4}\left( p_{2}-2;\Delta \right).
\end{eqnarray*}%
\end{theorem}

The risk comparison of biased full model, sub model, pretest and shrinkage estimators have been discussed in \cite{yuzbasi-et-al2017}. Since we get similar results here, the details of discussion are omitted. In order to compare the relative risks of estimators, we implemented a Monte Carlo simulation study as the following section.

\section{Simulation}
In this section, we consider a Monte Carlo simulation
to evaluate the performance of the suggested estimators. 
The response is obtained from the following model:
\begin{equation}
y_{i}=x_{1i}\beta _{1}+x_{2i}\beta _{2}+...+x_{pi}\beta _{p}+\varepsilon _{i},\ i=1,\ldots,n,  \label{sim.mod}
\end{equation}%
where $\varepsilon _{i}$ are i.i.d. $%
\mathcal{N}\left( 0,1\right)$, and the design matrix is generated from a multivariate normal distribution with zero mean and covariance matrix $\bm{\Sigma }_{x}$. Here, we consider that the off-diagonal elements of the covariance
matrix are considered to be equal to $\rho$. Furthermore, we consider the condition number (CN) value, which is defined as the ratio of the largest eigenvalue to the smallest eigenvalue of matrix $\bX^{\top}\bX$, to assess the multicollinearity. \citet{belsley1991} suggest that the data has multicollinearity if the CN value is larger than $30$.

\begin{itemize}
\item We consider the sample size $n=50,100$
\item $\rho=0.3, 0.6, 0.9$
\item We also consider that the regression coefficients are set $\bbeta=\left( \bbeta%
_{1}^{\top},\bbeta_{2}^{\top}\right)^{\top} =\left( \bbeta_{1}^{\top},%
\bm{0}_{p_2}^{\top}\right)^{\top}$ with $\bbeta_{1}=(\underset{p_1}{\underbrace{1,\dots1}})^{\top}$, where $\bm{0}_{p_2}$ is the zero vector with dimension $p_2$
\item In order to investigate the behaviour of the estimators, we define $\Delta^{\ast}=\left\Vert \boldsymbol{\beta -\beta }_{0}\right\Vert \ge 0$, where $\bbeta_{0}=\left( \bbeta_{1}^{\top},\boldsymbol{0}_{p_2}^{\top}\right)^{\top}$ and $\left\Vert \cdot \right\Vert $ is the Euclidean norm. To clarify this equation, one may write 
$\bbeta=(\underset{p_1}{\underbrace{1,\dots1}}, \Delta^{\ast}, \underset{p_2-1}{\underbrace{0,\dots,0}})^{\top}$ to generate response. If $\Delta^{\ast}=0$, then the null hypothesis is true, otherwise it is not

\item The number of predictor variables: $\left (p_1,p_2  \right ) \in \left \{ (5,5),(5,10),(5,15),(5,30) \right \}$
\item Each realization was repeated 1000 times to calculate the MSE of suggested estimators
\item $\alpha$ is taken as $0.05$
\end{itemize}

\begin{table}[H] 
\centering 
  \caption{Simulated RMSE results when \\ ($p_1=5$, $p_2=15$ and $ n= 100$)} 
  \label{Tab1} 
  \begin{adjustbox}{width=1\textwidth}
\begin{tabular}{rrrrrrrrrrrrr}
\toprule
  &&\multicolumn{2}{c}{$\rho=0.3$} &&& \multicolumn{2}{c}{$\rho=0.6$} &&& \multicolumn{2}{c}{$\rho=0.9$} \\
\cmidrule(lr){2-5} \cmidrule(lr){6-9} \cmidrule(lr){10-13}
$\Delta^{\ast}$ & $\hbbeta_{1}^{\rm LSM}$ & $\hbbeta_{1}^{\rm LPT}$ & $\hbbeta_{1}^{\rm LS}$ & $\hbbeta_{1}^{\rm LPS}$ &
 $\hbbeta_{1}^{\rm LSM}$ & $\hbbeta_{1}^{\rm LPT}$ & $\hbbeta_{1}^{\rm LS}$ & $\hbbeta_{1}^{\rm LPS}$ &
 $\hbbeta_{1}^{\rm LSM}$ & $\hbbeta_{1}^{\rm LPT}$ & $\hbbeta_{1}^{\rm LS}$ & $\hbbeta_{1}^{\rm LPS}$ \\ 
\midrule
0.000 & 0.733 & 0.763 & 0.774 & 0.760 & 0.715 & 0.749 & 0.755 & 0.743 & 0.538 & 0.585 & 0.587 & 0.585 \\ 
  0.250 & 0.753 & 0.851 & 0.800 & 0.794 & 0.751 & 0.830 & 0.791 & 0.783 & 0.571 & 0.620 & 0.621 & 0.615 \\ 
  0.500 & 0.814 & 0.987 & 0.857 & 0.856 & 0.863 & 1.008 & 0.864 & 0.863 & 0.596 & 0.677 & 0.650 & 0.644 \\ 
  0.750 & 0.928 & 1.001 & 0.912 & 0.912 & 1.061 & 1.011 & 0.919 & 0.919 & 0.652 & 0.775 & 0.702 & 0.697 \\ 
  1.000 & 1.075 & 1.000 & 0.941 & 0.941 & 1.335 & 1.000 & 0.946 & 0.946 & 0.734 & 0.900 & 0.763 & 0.755 \\ 
  1.250 & 1.288 & 1.000 & 0.959 & 0.959 & 1.656 & 1.000 & 0.965 & 0.965 & 0.813 & 1.008 & 0.801 & 0.797 \\ 
  1.500 & 1.561 & 1.000 & 0.971 & 0.971 & 2.071 & 1.000 & 0.973 & 0.973 & 0.956 & 1.073 & 0.856 & 0.854 \\ 
  2.000 & 2.195 & 1.000 & 0.983 & 0.983 & 3.172 & 1.000 & 0.985 & 0.985 & 1.247 & 1.046 & 0.916 & 0.916 \\ 
  4.000 & 6.792 & 1.000 & 0.996 & 0.996 & 10.438 & 1.000 & 0.996 & 0.996 & 3.300 & 1.000 & 0.984 & 0.984 \\ 
\bottomrule
\end{tabular}
\end{adjustbox}
\end{table}

All computations were conducted using the statistical package \citet{R2010}.  The performance of one of the suggested estimator was evaluated by using MSE criterion. Also, the relative mean square efficiency
(RMSE) of the $\bbeta_{1}^{\blacktriangle }$ to the
 $\hbbeta_{1}^{\rm LFM}$
is indicated by%
\begin{equation*}
\textnormal{RMSE}\left( \hbbeta_{1}^{\rm LFM}:\bbeta%
_{1}^{\blacktriangle }\right) =\frac{\textnormal{MSE}\left( \bbeta_{1}^{\blacktriangle}\right) }{\textnormal{MSE}\left( \hbbeta%
_{1}^{\rm LFM}\right) },
\end{equation*}
where $\bbeta_{1}^{\blacktriangle }$ is one of the listed estimators. If the RMSE of an estimators smaller than one, then it indicates superior to the full model estimator.

\begin{figure}[!htbp] 
\centering
\includegraphics[width=15cm,height=19cm]{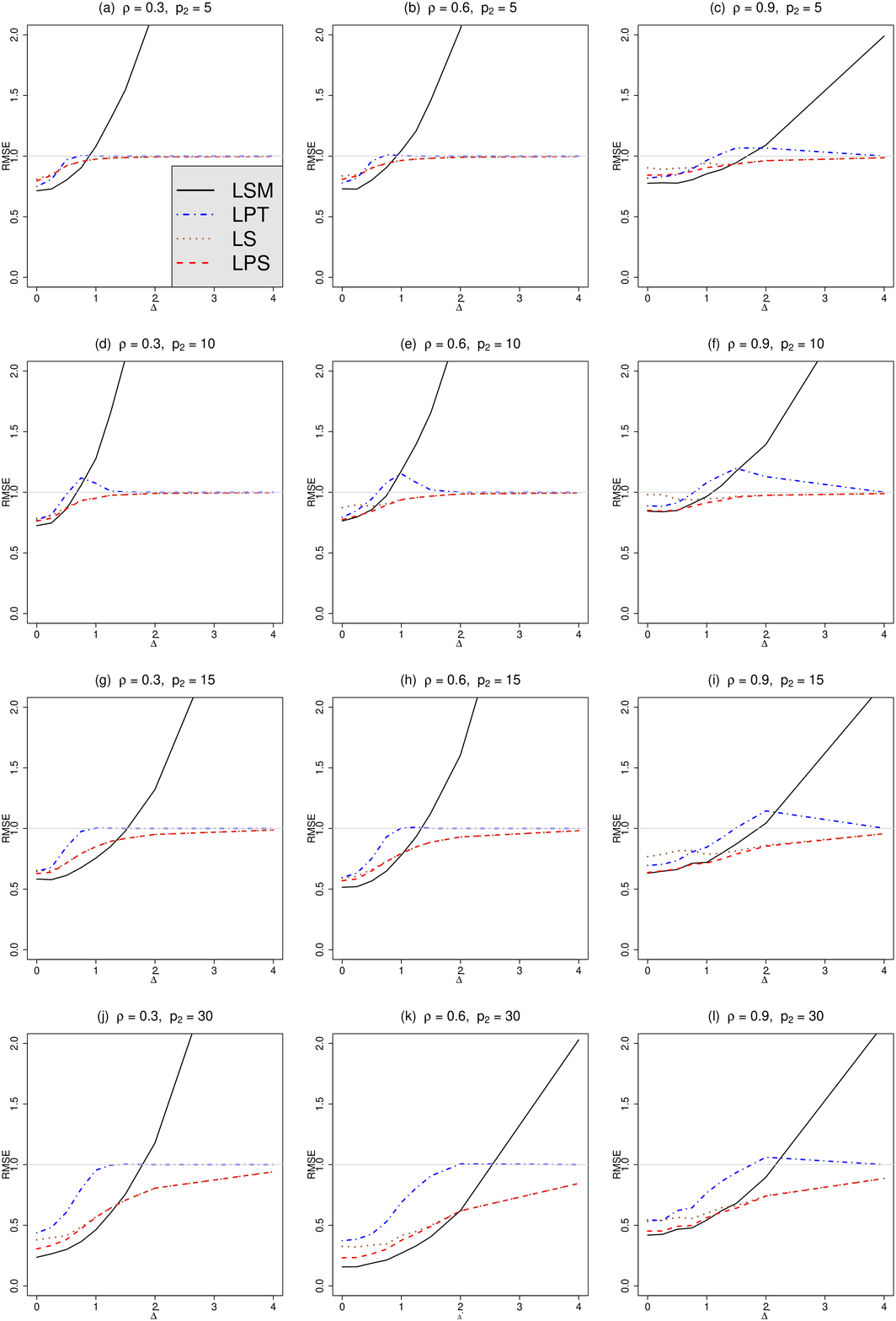}
\caption{ RMSE of the estimators as a function of the
non-centrality parameter $\Delta^{\ast}$ when $n=50$ and ${p}_1=5$.
 \label{Fig:RMSE:n50}}
\end{figure}

For the sake of brevity, we report the results for $n=100$, $p_1=5$ and $p_2=15$  with the different values of $\rho$ are shown in Table \ref{Tab1}. Furthermore, we also plotted RMSEs against $\Delta ^{\ast }$ for easier comparison in Figures \ref{Fig:RMSE:n50} and \ref{Fig:RMSE:n100}. 

\begin{figure}[!htbp] 
\centering
\includegraphics[width=15cm,height=20cm]{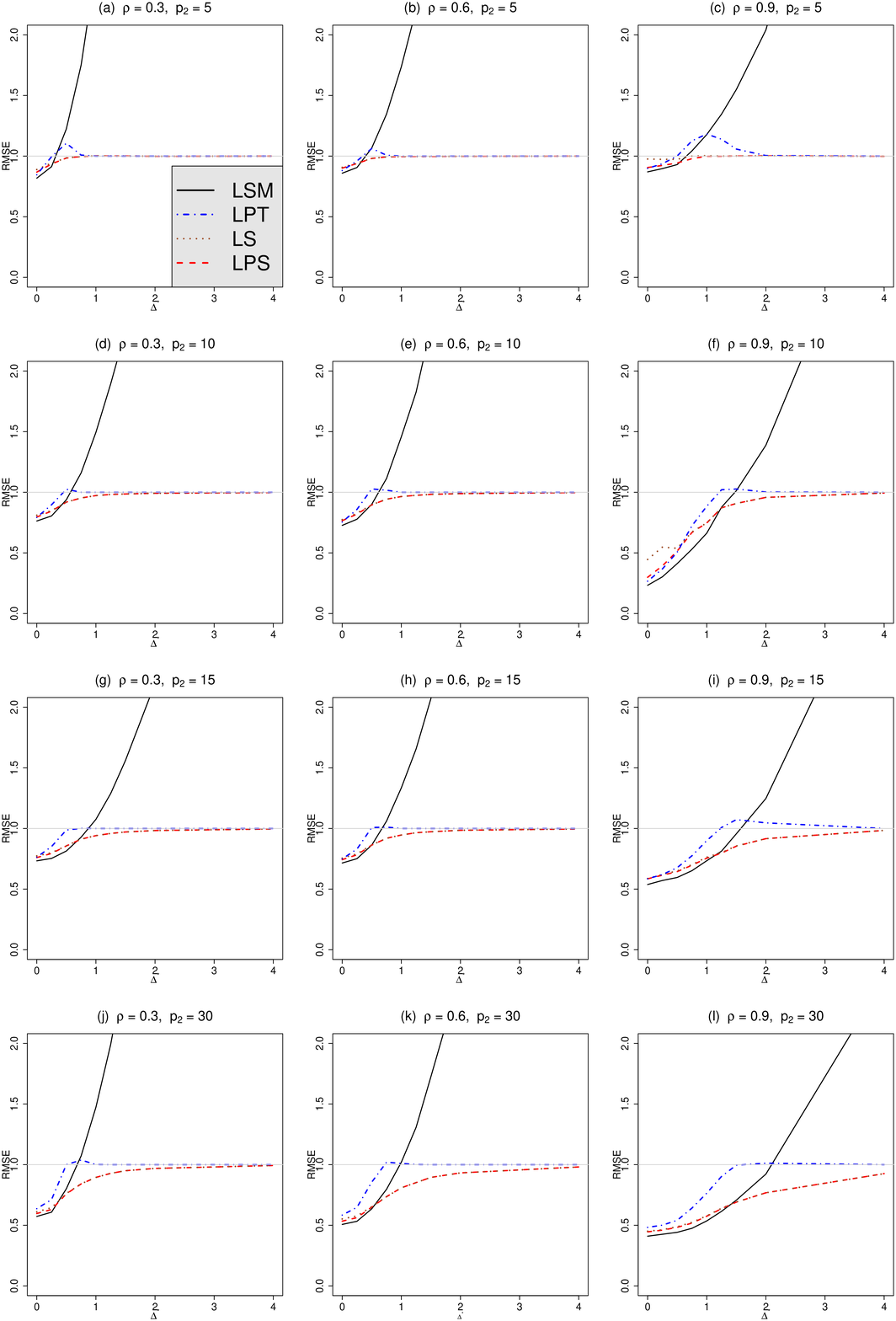}
\caption{ RMSE of the estimators as a function of the
non-centrality parameter $\Delta^{\ast}$ when $n=100$ and ${p}_1=5$.
 \label{Fig:RMSE:n100}}
\end{figure}

In summary, when $\Delta ^{\ast }=0$, i.e. the null hypothesis is true, not surprisingly the LSM is superior to all estimators, since it has the smallest RMSE. In contrast, the LSM does not perform well when the value of $\Delta ^{\ast }$ increases. Also, the RMSE of the LPT is smaller than the RMSEs of LS and LPS for small values of $\Delta ^{\ast }$. For the intermediate values $\Delta ^{\ast }$, however, the RMSE of LPT may lose its efficiency, even worse than the LFM. Finally, the larger values of $\Delta ^{\ast }$, the RMSEs of LPT approaches to one. As it can be shown that the performance of LPS always outshines LS for all values of $\Delta ^{\ast }$. Again, the RMSEs of LS and LPS goes to one for large values of $\Delta ^{\ast }$.

\subsection{Comparisons with $L_1$ estimators}

In Table \ref{T2_all}, we compare our listed estimators with  LSE and some penalty estimators, namely Ridge, Lasso, aLasso, SCAD and MCP when $\rho=0.3, 0.6, 0.9$, $n=30, 80$, $p_1=4$ and $p_2=4, 10$. According to Table \ref{T2_all},  shrinkage estimators outshine all others.

\begin{table}[!htbp] 
\caption{RMSE comparions for ${p}_1=4$.}
\label{T2_all}
\centering
\begin{adjustbox}{width=1\textwidth}
\begin{tabular}{llrrllllllllll}
\toprule
$\rho$ &$n$& $p_2$ &
$\hbbeta_{1}^\textnormal{LSE}$ &
$\hbbeta_{1}^\textnormal{Ridge}$ &
$\hbbeta_{1}^\textnormal{LSM}$ &
$\hbbeta_{1}^\textnormal{LPT}$ &
$\hbbeta_{1}^\textnormal{LS}$ &
$\hbbeta_{1}^\textnormal{LPS}$ &
$\hbbeta_{1}^\textnormal{Lasso}$ &
$\hbbeta_{1}^\textnormal{aLasso}$ &
$\hbbeta_{1}^\textnormal{SCAD}$ &
 \\
\midrule



0.3&30&4&1.06496&1.02669&0.00085&0.00089&0.00087&0.00088&1.06045&1.16279&1.34048\\
&&10&0.00099&0.00085&0.00035&0.00056&0.00050&0.00049&0.00072&0.00064&0.86505\\
&80&4&0.00097&0.00096&0.00085&0.00086&0.00093&0.00090&1.00100&0.00096&1.01112\\
&&10&0.00099&0.00098&0.00075&0.00076&0.00083&0.00080&0.00092&0.00089&1.02249\\ \\

0.6&30&4&0.00096&0.00088&0.00062&0.00067&0.00078&0.00076&0.00088&1.00402&1.22100\\
&&10&0.00098&0.00073&0.00041&0.00063&0.00054&0.00053&0.00061&0.00066&1.18906\\
&80&4&0.00097&0.00095&0.00075&0.00076&0.00085&0.00084&0.00092&0.00092&1.12233\\
&&10&0.00097&0.00095&0.00064&0.00070&0.00068&0.00068&0.00091&0.00086&0.89928\\ \\

0.9&30&4&1.04712&0.00088&0.00062&0.00069&0.00088&0.00078&0.00094&1.12360&1.48810\\
&&10&0.00094&0.00055&0.00028&0.00056&0.00047&0.00045&0.00060&0.00065&1.05708\\
&80&4&0.00099&0.00095&0.00073&0.00080&0.00086&0.00084&0.00097&1.16959&1.61812\\
&&10&0.00100&0.00094&0.00065&0.00071&0.00076&0.00072&0.00093&1.06270&1.72414\\

\bottomrule
\end{tabular}%
\end{adjustbox}
\end{table}

\section{Application to State data}

We consider the State data set which is available by default in R. This data set is related to the 50 states of the United States of America. We list all variables in Table~\ref{Tab:variables}. We also consider the life expectancy as the response.

\begin{table}[H]
\small
\centering
\begin{tabular}{|l|l|}
\hline
\textbf{Dependent Variable} &\\
life.exp & life expectancy in years (1969--71)\\
\hline 
\textbf{Covariates}  & \\
population & population estimate as of July 1, 1975\\
income &per capita income (1974)\\
illiteracy & illiteracy (1970, percent of population)\\
murder & murder and non-negligent manslaughter \\ &rate per 100.000 population (1976)\\
hs.grad & percent high-school graduates (1970)\\
frost & mean number of days with minimum temperature  \\ &below freezing (1931--1960) in capital or large city\\
area & land area in square miles\\
\hline 
\end{tabular}
\caption{Lists and Descriptions of Variables
\label{Tab:variables}}
\end{table}

In Figure~\ref{Fig:cor_plot}, we plot the coefficients of correlation among covariates. We also show the degree of correlation with colours, and the cells which has not any colour indicate that it is not significant with $\alpha=0.05$. According to this figure, there are strong relationship among some predictors. This situation encouraged us to use our suggested estimator since they perform superiorly.

\begin{figure}[H]
\centering
\includegraphics[width=10cm,height=10cm]{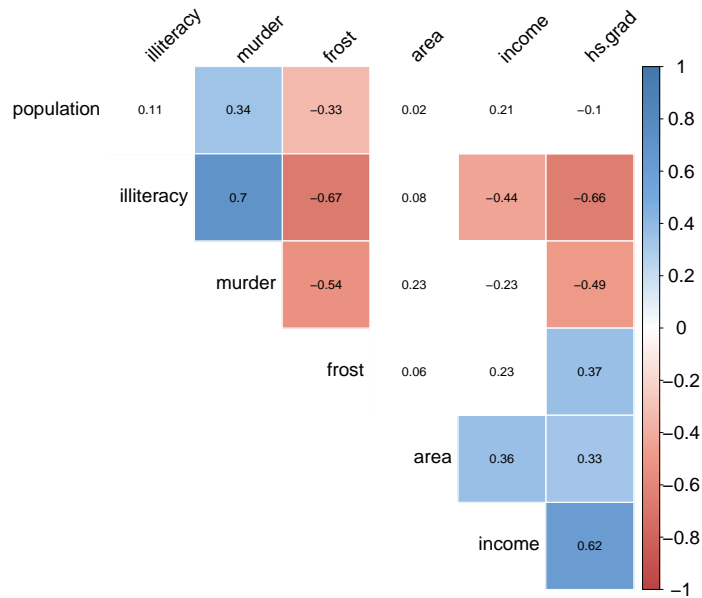}
\caption{Correlations among predictors
 \label{Fig:cor_plot}}
\end{figure}

In any application, if we do not have any prior information about covariates whether they are significantly important or not, one might do stepwise or variable selection techniques to select the best subsets. In this study, we use AIC method, we find that income, illiteracy and area variables do not significantly explain the response variable, and these covariates may be ignored. Hence, we fit the sub-model with the help of this auxiliary information, and the full and candidate sub-models are given in Table~\ref{Tab:fittings}.

\begin{table}[H]
\footnotesize
\centering
\begin{tabular}{ll}
\hline
\textbf{Models} & \textbf{Formulas} \\
\hline
Full model &  life.exp = $\beta_0+\beta_1$(population) $+\beta_2$(income)$+\beta_3$(illiteracy)$+\beta_4$(murder)$+\beta_5$(hs.grad)\\ &$+\beta_6$(frost)$+\beta_7$(area)\\
Sub-Model &  life.exp = $\beta_0+\beta_1$(population) $+\beta_4$(murder)$+\beta_5$(hs.grad)$+\beta_6$(frost)\\
\hline
\end{tabular}
\caption{Fittings of full and sub-model
\label{Tab:fittings}}
\end{table}

\begin{figure}[H]
\centering
\includegraphics[width=12cm,height=6cm]{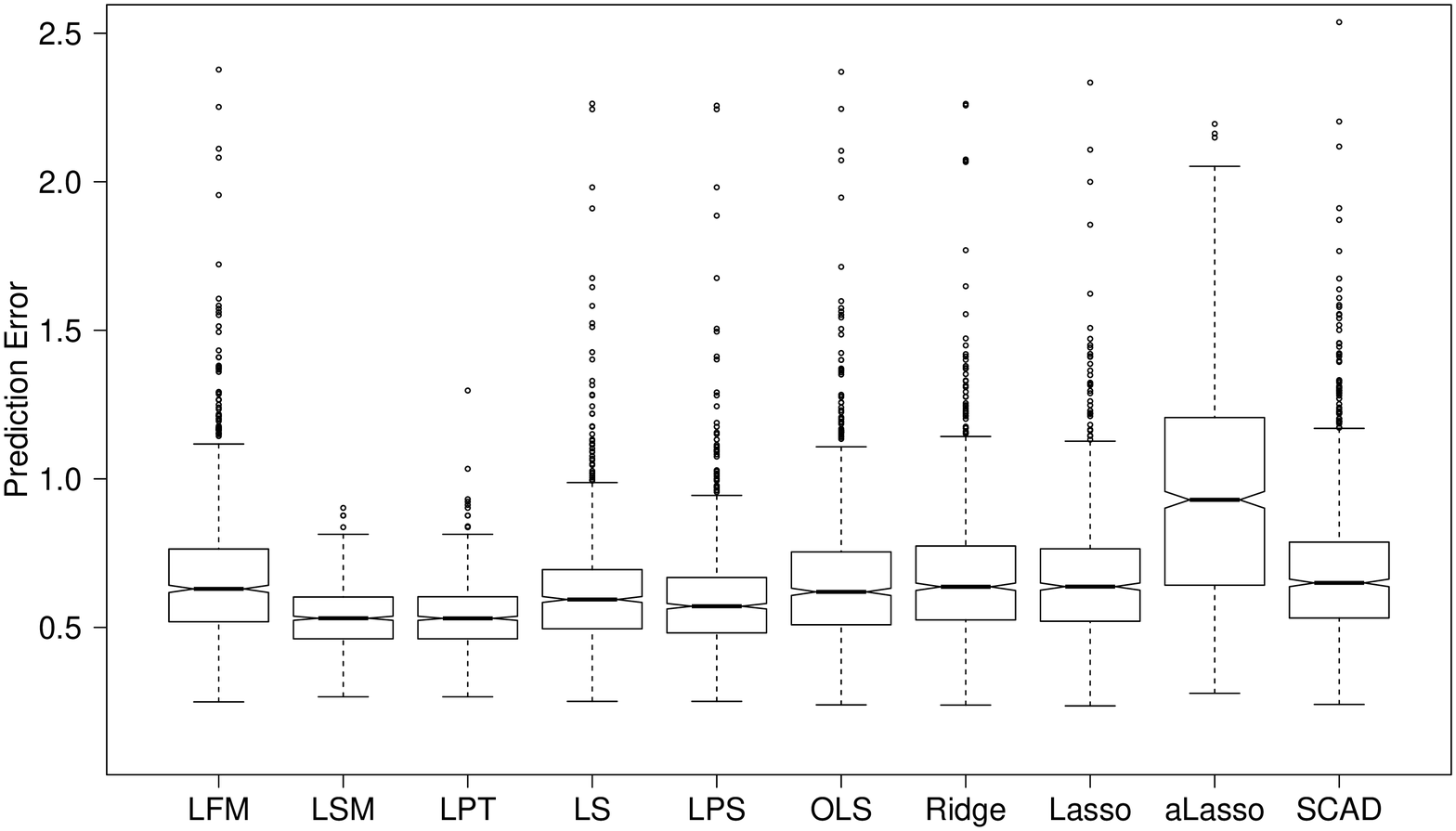}
\caption{Prediction errors of listed estimators based on bootstrap simulation
 \label{Fig:boxplot:PE}}
\end{figure}

Our results are based on bootstrap samples resampled $1000$ times. Since there is no noticeable variation for larger number of replications, we did not consider further values. The average prediction errors were calculated via 10-fold CV for each bootstrap replicate. The predictors were first standardized to have zero mean and unit standard deviation before fitting the model. To evaluate the performance of the suggested estimators, we calculate the predictive error (PE) of an estimator.  In Figure~\ref{Fig:boxplot:PE}, we plot the prediction errors versus the listed estimators.

Furthermore, we define the relative predictive error (RPE) of $\hbbeta^{\ast}$ in terms of the full model Liu regression estimator $\hbbeta^{\rm LFM}$ to ease comparison as follows
\begin{equation*}
\text{RPE}\left(\hbbeta^{\ast} \right) =\frac{\text{PE}(\hbbeta^{\ast})}
{\text{PE}(\hbbeta^{\rm LFM} )},
\end{equation*}%
where $\hbbeta^{\ast}$ can be any of the listed estimators. If the RPE is smaller than one, it indicates the superiority to LFM. Table~\ref{Tab:estimations} reveals the RPE of the listed estimators. According to Table~\ref{Tab:estimations}, the sub-model estimator has the smallest RPE since it is computed based on the assumption that the selected sub-model is the true model. As expected, due to the presence of multicollinearity, the performance of both Liu-type shrinkage and pretest estimators are better than the estimators based on $L_1$ criteria. Thus, the data analysis corroborates with our simulation and theoretical findings.

\begin{table}[H]
\centering
\begin{tabular}{rrrrrrr}
  \hline
&(intercept) & Population & Murder & HS.Grad & Frost &RPE\\ 
  \hline
LFM&70.894 & 0.202 & -1.077 & 0.335 & -0.282 & 1.000 \\ 
  &0.016 & -0.029 & 0.034 & -0.061 & 0.016 \\ 
  &0.117 & 0.153 & 0.181 & 0.223 & 0.175 \\ 
 LSM& 70.881 & 0.234 & -1.088 & 0.385 & -0.283 &0.788\\ 
  &0.003 & 0.010 & 0.020 & 0.008 & 0.026 \\ 
  &0.104 & 0.117 & 0.134 & 0.134 & 0.161 \\ 
  LPT&70.882 & 0.232 & -1.088 & 0.383 & -0.283 &0.794\\ 
  &0.004 & 0.008 & 0.020 & 0.007 & 0.026 \\ 
  &0.105 & 0.118 & 0.136 & 0.136 & 0.162 \\ 
  LS&70.889 & 0.216 & -1.091 & 0.350 & -0.286 &0.917\\ 
  &0.011 & -0.014 & 0.021 & -0.043 & 0.013 \\ 
  &0.110 & 0.133 & 0.158 & 0.176 & 0.171 \\ 
  LPS&70.889 & 0.215 & -1.086 & 0.353 & -0.283 &0.883\\ 
  &0.011 & -0.015 & 0.025 & -0.040 & 0.016 \\ 
  &0.110 & 0.132 & 0.153 & 0.169 & 0.162 \\ 
  LSE&70.879 & 0.231 & -1.112 & 0.395 & -0.298 &1.005\\ 
  &0.016 & -0.029 & 0.034 & -0.061 & 0.016 \\ 
  &0.117 & 0.153 & 0.181 & 0.223 & 0.175 \\ 
  Ridge&70.893 & 0.119 & -0.816 & 0.312 & -0.188&0.998 \\ 
  &0.015 & -0.033 & 0.058 & -0.043 & 0.048 \\ 
  &0.115 & 0.124 & 0.190 & 0.145 & 0.136 \\ 
  Lasso&70.891 & 0.128 & -0.942 & 0.281 & -0.182 &0.984\\ 
  &0.012 & -0.055 & 0.090 & -0.063 & 0.051 \\ 
  &0.113 & 0.119 & 0.190 & 0.177 & 0.156 \\ 
  aLasso&70.888 & 0.096 & -0.464 & 0.571 & -0.079 &1.385 \\ 
  &0.010 & 0.104 & -0.464 & -0.223 & -0.079 \\ 
  &0.129 & 0.167 & 0.364 & 0.233 & 0.133 \\ 
  SCAD&70.888 & 0.141 & -1.057 & 0.259 & -0.201&1.026 \\ 
  &0.009 & -0.083 & 0.051 & -0.117 & 0.107 \\ 
  &0.116 & 0.140 & 0.175 & 0.219 & 0.187 \\ 
   \hline
\end{tabular}
\caption{
Estimate (first row), Bias (second row) and standard error (third row) for significant coefficients for the state data. The RPE column gives the relative efficiency based on bootstrap simulation with respect to the LFM. 
\label{Tab:estimations}}
\end{table}

\section{Conclusions}
In this paper, we combined the pre-test estimator and Stein-type estimator with the Liu regression method in order to obtain a better estimators in the linear regression model when the parameter vector $\bbeta$ is partitioned into two parts, namely, the main effects $\bbeta_1$ and the nuisance effects $\bbeta_2$ such that $\bbeta=\left(\bbeta_1, \bbeta_2 \right)$. Thus, our main interest is to estimate $\bbeta_1$ when $\bbeta_2$ is close to zero. Therefore, we conduct a Monte Carlo simulation study to evaluate the relative efficiency of the suggested estimators and also we present a real data application. According to  both the results of the simulation and real application, we conclude that our estimators have better performance than LSE, ridge and the estimators based on $L_1$ criteria.

\section*{Acknowledgement}
This research is supported by Necmettin Erbakan University, Scientific Research Projects Unit, Project No: 171215001.

\end{document}